\documentclass[11pt]{article}
\usepackage{microtype}
\usepackage[pdfpagelabels,breaklinks,colorlinks,linkcolor=blue,citecolor=blue,urlcolor=blue]{hyperref}
\usepackage[textwidth=6.5in,textheight=8.6in]{geometry}
\usepackage{placeins}
\usepackage{epsfig}
\usepackage{latexsym}
\usepackage{amssymb}
\usepackage{amsmath}
\usepackage{algorithm}
\usepackage{amsthm}
\usepackage{caption}
\usepackage{amsfonts}
\usepackage{float}
\usepackage{paralist,dsfont}
\usepackage{enumitem}
\usepackage{enumerate}
\usepackage{cases}
\usepackage[normalem]{ulem}
\usepackage{color,xcolor}

\DeclareTextFontCommand{\textbfit}{%
  \fontseries\bfdefault 
  \itshape
}

\usepackage{adjustbox}

\colorlet{blue}{blue!80!black}
\colorlet{red}{red!80!black}
\colorlet{green}{green!50!black}

\newtheorem{assumption}{Assumption}[section]
\newtheorem{theorem}{Theorem}[section]
\newtheorem{proposition}[theorem]{Proposition}
\newtheorem{lemma}[theorem]{Lemma}

\theoremstyle{definition}
\newtheorem{definition}[theorem]{Definition}
\newtheorem{example}[theorem]{Example}

\theoremstyle{remark}
\newtheorem{remark}[theorem]{Remark}
\numberwithin{equation}{section}

\DeclareMathOperator*{\argmin}{arg\,min}
\newcommand{\dist}{{\sf dist}}
\newcommand{\ds}{\displaystyle}
\newcommand{\diag}{{\sf Diag\,}}

\renewcommand{\int}{{\sf int\,}}
\newcommand{\bdry}{{\sf bdry\,}}

\def\[{\begin{equation}}
\def\]{\end{equation}}
\def\cA{{\mathcal A}}

\def\cC{{\mathcal C}}

\def\cT{{\mathcal T}}
\def\cH{{\mathcal H}}
\def\cU{{\mathcal U}}

\def\cQ{{\mathcal Q}}
\def\cM{{\mathcal M}}

\def\cW{{\mathcal W}}
\def\bW{{\mathbf W}}
\def\cS{{\mathcal S}}

\def\cV{{\mathcal V}}
\def\bV{{\mathbf V}}
\def\cI{{\mathcal I}}
\def\cJ{{\mathcal J}}
\def\cX{{\mathcal X}}
\def\cZ{{\mathcal Y}}
\def\cZ{{\mathcal Z}}

\def\cL{{\mathcal L}}
\def\bfx{{\mathbf x}}
\def\aff{{\sf aff\,}}
\def\app{{\sf app\,}}
\def\lin{{\sf lin\,}}

\renewcommand\Re{{\mathds R}}

\title{\Large \bf A Unified Convergence Analysis of a Second-Order Method of Multipliers for Nonlinear Conic Programming}

\author{Liang Chen\thanks{School of Mathematics, Hunan University, Changsha 410082, P.R. China (\url{chl@hnu.edu.cn}).
This author was supported in part by the National Natural Science Foundation of China (11801158), the Hunan Provincial Natural Science Foundation of China (2019JJ50040), and the Fundamental Research Funds for the Central Universities in China.}\quad
{Junyuan Zhu}\thanks{School of Mathematics, Hunan University, Changsha 410082, P.R. China (\url{jyz@hnu.edu.cn}).} \quad
{Xinyuan Zhao}\thanks{School of Mathematics, Beijing University of Technology, Beijing 100124, P. R. China (\url{xyzhao@bjut.edu.cn}). This author was supported in part by the National Natural Science Foundation of China (11871002) and the General Program of Science and Technology of Beijing Municipal Education Commission.}}

\date{\today}

\begin{document}
\maketitle

\begin{abstract}
In this paper, we accomplish a unified convergence analysis of a second-order method of multipliers (i.e., a second-order augmented Lagrangian method) for solving the conventional nonlinear conic optimization problems.
Specifically, the algorithm that we investigated incorporates a specially designed nonsmooth (generalized) Newton step to furnish a second-order update rule for the multipliers.
We first show in a unified fashion that under a few abstract assumptions, the proposed method is locally convergent and possesses a (nonasymptotic) superlinear convergence rate, even though the penalty parameter is fixed and/or the strict complementarity fails.
Subsequently, we demonstrate that, for the three typical scenarios, i.e., the classic nonlinear programming, the nonlinear second-order cone programming, and the nonlinear semidefinite programming, these abstract assumptions are nothing but exactly the implications of the iconic sufficient conditions that were assumed for establishing the Q-linear convergence rates of the method of multipliers without assuming the strict complementarity.
\end{abstract}

\medskip
{\small
\begin{center}
\parbox{0.95\hsize}{{\bf Keywords.}\;
second-order method of multipliers, augmented Lagrangian method, convergence rate,
generalized Newton method, second-order cone programming, semidefinite programming
}
\end{center}
\begin{center}
\parbox{0.95\hsize}{{\bf AMS Subject Classification [2000].}\;
Primary 65K05, 49J52; Secondary 90C22, 26E25}
\end{center}}

\section{Introduction}
Let $\cX$ and $\cZ$ be two finite-dimensional real Hilbert spaces each endowed with an inner product $\langle \cdot,\cdot\rangle$ and its induced norm $\|\cdot\|$.
Consider the following general nonlinear conic optimization problem
\begin{equation}
\label{op}
\min_{x\in\cX}\,
f(x)
\quad\mbox{s.t.}\quad
\ h(x)=0,\
g(x)\in K,
\end{equation}
where
$f: \cX\to\Re$, $h:\cX\to \Re^{m}$ and $g:\cX\to\cZ$ are three twice continuously differentiable functions, and $K$ is a closed convex self-dual cone in $\cZ$, i.e., it holds that
$$
K\equiv K^*:=\left\{\mu\, |\,\langle \mu,\mu'\rangle\ge 0,\ \forall \mu'\in K\right\}.
$$

Let $c>0$ be the given penalty parameter.
The augmented Lagrangian function of problem (\ref{op}) is defined by
\[
\label{alf}
\begin{array}{ll}
\ds
\cL_c(x,\lambda,\mu):
=f(x)-\langle\lambda, h(x)\rangle+\frac{c}{2}\|h(x)\|^2+\frac{1}{2c}\big(\|\Pi_{ K}(\mu-cg(x))\|^2-\|\mu\|^2\big),\qquad
\\[3mm]
\ds
\hfill\forall (x,\lambda,\mu)\in \cX\times\Re^{m}\times \cZ,
\end{array}
\]
where $\Pi_{K}(\cdot)$ denotes the metric projection operator onto $K$.

The \emph{method of multipliers} (also known as the augmented Lagrangian method)
for solving problem (\ref{op}) can be stated, in brief, as follows.
Let $c_0>0$ be a given constant and $(\lambda^0,\mu^0)\in \Re^{m}\times K$ be the initial estimate of the multipliers. For $k=0,1,\ldots$, determine $x^{k+1}\in\cX$ that minimizes the augmented Lagrangian function $\cL_{c_k}(x, \lambda^k, \mu^k)$ with respect to $x$, and then update the multipliers via
\[
\label{update}
\left\{
\begin{array}{l}
\lambda^{k+1}:=\lambda^k-c_k h(x^{k+1}),
\\[2mm]
\mu^{k+1}:=\Pi_{K}\big(\mu^k-c_k g(x^{k+1})\big),
\end{array}
\right.
\]
thereafter, update the penalty parameter such that $c_{k+1}\ge c_{k}$.

This method was initiated in Hestenes \cite{hestenes69} and Powell \cite{powell} for solving equality constrained nonlinear programming (NLP) problems,
and was generalized in Rockafellar \cite{roc73} for NLP problems with inequality constraints.
Nowadays, the augmented Lagrangian method has been serving as the backbone of many optimization solvers, such as \cite{lancelot,sdpnal,sdpnalp,lassonal}, which are readily available for both academic and commercial users.
The literature of the augmented Lagrangian method is considerably abundant, so that we are only able to review the references on analyzing and improving its convergence rate, which are the most relevant ones to this work.

The linear convergence rates of the method of multipliers for NLP problems has been extensively discussed in the literature, such as Powell \cite{powell}, Rockafellar \cite{roc73,roc732}, Tretyakov \cite{tre73}, Bertsekas \cite{ber76}, Conn \cite{conn}, Ito and Kunisch \cite{ito}, and Contesse-Becker \cite{cont93}, under various assumptions.
One may refer to the monograph of Bertsekas \cite{berbook} and the references therein for a thorough discussion on the augmented Lagrangian method for NLP problems.
For nonlinear second-order cone programming (NLSOCP) problems,
the local linear convergence rate of the augmented Lagrangian method has been analyzed in Liu and Zhang \cite{liuzhang1}, and the corresponding results were further improved in Liu and Zhang \cite{liuzhang2}.
For nonlinear semidefinite programming (NLSDP) problems, the local linear rate of convergence was analyzed by Sun et al. \cite{sun07}, and, importantly, the strict complementarity is not necessary for getting this result.
Here, we should mention that, in the vast majority of the above mentioned references, the convergence rates of the method of multipliers could be asymptotically superlinear, provided that the penalty parameter sequence $\{c_k\}$ tends to infinity as $k\to\infty$.

In the seminal work \cite{rockafellar} of Rockafellar, it was observed that, in the convex programming setting, the method of multipliers can be viewed as a proximal point algorithm applied to the dual.
On the other hand, without assuming the convexity, one can locally interpret the iteration procedure \eqref{update} for updating the multipliers as a first-order steepest ascent step applied to a certain concave (dual) problem.
Consequently, given the linear convergence rates of the method of multipliers mentioned above,
it is natural to pose the question that if one can design a second-order method of multipliers in the sense that a second-order update form of the multipliers, instead of \eqref{update}, is incorporated in the method of multipliers for calculating $\lambda^{k+1}$ and $\mu^{k+1}$.
Meanwhile, it is also interesting to know if the corresponding superlinear convergence rates are achievable, even when the sequence $\{c_k\}$ is bounded and/or the strict complementarity fails.

In fact, several attempts have been made to resolve this question for quite a long time.
For NLP problems without inequality constraints, Buys \cite{buys} first introduced a second-order procedure for updating the multipliers, which can be viewed as an application of the classic Newton method to a certain dual problem.
Such a procedure was later investigated and refined by Bertsekas \cite{ber76b,ber78,berbook}.
Later, Brusch \cite{brusch} and Fletcher \cite{fletcher} independently proposed updating rules of the multipliers via the quasi-Newton method, and this was also studied in Fontecilla et al. \cite{font}.
For NLP problems with inequality constraints, the augmented Lagrangian function is no longer twice continuously differentiable. Therefore, it is much more difficult to design a second-order method of multipliers for these problems.
The first progress towards this direction was made by Yuan \cite{yuan14}, in which the updating rule is based on the classic Newton method, but the corresponding convergence result still has not been formally provided.
More recently, Chen and Liao \cite{chenlalm} proposed a second-order augmented Lagrangian method for solving NLP problems with inequality constraints. This algorithm is based on a specially designed generalized Newton method and, the corresponding local convergence, as well as the superlinear convergence rate, was established. At the same time, extensive numerical study of the second-order method of multipliers has also been conducted for convex quadratic programming \cite{Bueno}, and the numerical results suggest that the second-order update formula of the multipliers, at least locally, can be much better than the usual first-order update rule \cite[Section 6]{Bueno}.

Motivated by the references mentioned above, in this paper, we would like to approach the problem of how to extend the second-order method of multipliers in \cite{chenlalm} from NLP problems with inequality constraints to the more general conic optimization cases, i.e., the more complicated NLSOCP and NLSDP problems.
We should emphasize that, as can be observed later, such an extension is highly nontrivial, inasmuch as the required variational analysis for the non-polyhedral conic constraints is much more intricate than that for the NLP case, in which $K$ is a polyhedral convex set.

In the same vein as \cite{chenlalm}, we first use a specially designed generalized Newton method to deal with the nonsmoothness induced by the conic constraints, so as to construct a unified second-order method of multipliers for general nonlinear conic programming problems.
We analyze its convergence (locally convergent with a superlinear convergence rate even the strict complementarity fails and the penalty parameter is fixed) under certain abstract assumptions.
After that, for the NLP, NLSOCP and NLSDP cases, we show separately that these abstract assumptions can be implied by certain second-order optimality conditions plus a constraint qualification, which have been frequently used for establishing the Q-linear convergence rate of the method of multipliers.
Furthermore, when the conic constraints vanish, the algorithm studied in this paper automatically turns to the classic second-order method of multipliers for NLP problems with only equality constraints, and the convergence results in this paper are consistent with all the previous related studies.

The remaining parts of this paper are organized as follows.
In Section \ref{sec2}, we prepare the necessary preliminaries from the nonsmooth analysis literature for further discussions.
In Section \ref{sec3}, we propose a second-order method of multipliers for solving problem \eqref{op}, and establish its convergence under certain abstract assumptions.
In Sections \ref{sec4}, \ref{sec5} and \ref{sect6}, we separately specify the abstract assumptions in Section \ref{sec3} with respect to the NLP, NLSCOP and NLSDP cases, which constitutes the main contribution of this paper.
Section \ref{secnum} presents a simple numerical example to illustrate the superlinear convergence of the method studied in this work.
We conclude this paper in Section \ref{sec7}.

\section{Notation and preliminaries}
\label{sec2}
\subsection{Notation}
Let $\cU$ be an arbitrarily given finite-dimensional real Hilbert space endowed with an inner product $\langle\cdot,\cdot\rangle$ that induces the norm $\|\cdot\|$.
We use ${\mathds L}(\cU,\cU)$ to denote the linear space of all the linear operators from $\cU$ to itself.
For any given $\varepsilon>0$, we define the open $\varepsilon$-neighborhood in $\cU$ by
$${\mathds B}_{\varepsilon}(u)
:=\{u'\in \cU\mid \|u-u'\|<\varepsilon\},\quad\forall u\in\cU.$$

{ {Let $U$ be a nonempty closed (not necessarily convex) set in $\cU$ and $\lin U$ denote the lineality space of $U$, i.e., the largest linear space contained in $U$.}} We use $\aff U$ to denote the affine hull of $U$, i.e., the smallest linear subspace of $\cU$ that contains $U$.
Moreover, we define the distance from a point in $\cU$ to the set $U$ by
$$
\dist(u,U):=\inf_{u'\in U}\|u-u'\|,\quad \forall u\in \cU.
$$
We use $\cT_{U}^i(u)$ and $\cT_{U}(u)$ to denote the tangent cone and contingent cone of the set $U$ at $u\in\cU$, respectively, which are defined by
$$
\cT_{U}^i(u):=\left\{u'\in\cU\mid \dist(u+\tau u', U)=o(\tau),\ \forall \tau\ge 0\right\},
$$
and
$$
\cT_{U}(u):=\left\{u'\in\cU\mid \exists\, \tau_k\downarrow 0,\mbox{ such that } \dist(u+\tau_k u', U)=o(\tau_k)\right\}.
$$
If, additionally, $U$ is a convex set, one has that $\cT_{U}^i(u)=\cT_{U}(u),\,\forall u\in\cU$ and such tangent and contingent cones are simply called as the tangent cone of $U$ at $u$.
In this case, we define the projection operator to the set $U$ by
$$
\Pi_{U}(u)=\argmin_{u'\in U}\|u-u'\|,\quad \forall u\in \cU.
$$

Let $\cV$ and $\cW$ be the other two given finite-dimensional real Hilbert spaces each endowed with an inner product $\langle\cdot,\cdot\rangle$ and its induced norm $\|\cdot\|$. For any set $C\subseteq\cU\times\cV$, we define
$$
\pi_{\cU}(C):=\{
u\in\cU \mid (u,v)\in C
\}.
$$
Let $F:\cU\to\cW$ be an arbitrarily given function.
The directional derivative of $F(\cdot)$ at $u\in\cU$ along with a nonzero direction $\Delta u\in\cU$, if exists, is denoted by $F'(u;\Delta u)$.
$F$ is said to be directionally differentiable at $u\in\cU$ if $F'(u;\Delta u)$ exists for any nonzero $\Delta u\in \cU$.
If $F$ is a linear mapping, we use $F^*$ to denote its adjoint.
If $F$ is Fr\'echet-differentiable at $u\in\cU$,
we use $\cJ F(u)$ to denote the Fr\'echet-derivative of $F$ at this point and, in this case, we denote
$$\nabla F(u):=(\cJ F(u))^*.$$

Let $O$ be an open set in $\cU$.
If $F: O\to\cW$ is locally Lipschitz continuous on $O$, $F$ is almost everywhere Fr\'echet-differentiable in $O$, thanks to \cite[Definition 9.1]{va} and the Rademacher's theorem \cite[Theorem 9.60]{va}.
In this case, we denote the set of all the points in $O$ where $F$ is Fr\'echet-differentiable by $D_{F}$.
Then, the Bouligand-subdifferential of $F$ at $u\in O$, denoted by $\partial_B F(u)$, is defined by
$$
\partial_B F(u):=\left\{\lim_{k\to\infty}\cJ F(u^k)\mid \ \{u^k\}\subset D_{F},\, u^k\to u\, \right\}.
$$
Moreover, the Clarke's generalized Jacobian of $F$ at $u\in O$ is defined by
$$
\partial F(u):=\mbox{\sf conv}\{\partial_B F(u)\},
$$
i.e., the convex hull of $\partial_B F(u)$.

Let $G:\cU\times\cV\to \cW$ be an arbitrarily given function.
The partial Fr\'echet-derivative of $G$ to $u$ at $(u,v)\in\cU\times\cV$, if exists, is denoted by $\cJ_{u}G(u,v)$, and we define
$\nabla_u G(u,v):=(\cJ_{u} G(u,v))^*$.
Furthermore, if $G$ is twice Fr\'echet-differentiable at $(u,v)\in\cU\times\cV$, we denote
$$
\begin{array}{l}
\cJ^2G(u,v):=\cJ(\cJ G)(u,v),\quad
\cJ^2_{u,u}G(u,v):=\cJ_{u}(\cJ_{u} G)(u,v),\\[2mm]
\nabla^2G(u,v):=\cJ(\nabla G)(u,v),\quad
\nabla_{uu}^2G(u,v):=\cJ_u(\nabla_u G)(u,v).
\end{array}
$$

Finally, we mention that the notations in Sections \ref{sec4}, \ref{sec5}  and \ref{sect6} are isolated from each other, as the contents of these sections are independent.

\subsection{Nonsmooth Newton methods}
\label{ssn}
Recall that $\cU$ is a finite-dimensional real Hilbert space. Given a nonempty open set $O\subset\cU$. Let $F:O\to\cU$ be a locally Lipschitz continuous function and consider the following nonlinear equation
\[
\label{leqt}
F(u)=0.
\]
\begin{definition}[Semismoothness]
Suppose that $F:O\to\cU$ is locally Lipschitz continuous.
$F$ is said to be semismooth at a point $u\in O$ if $F$ is directionally differentiable at $u$ and, for any sequence $\{u^k\}\subset O$ that converging to $u$ with $V^k\in\partial F(u^k)$, i.e., the Clarke's generalized Jacobian of $F$ at $u^k$, it holds that
\begin{equation}
\label{ss1}
F(u^k)- F(u)-V^k(u^k-u)=o(\|u^k-u\|).
\end{equation}
Moreover, $F$ is said to be strongly semismooth at $u$ if $F$ is semismooth at $u$ but with \eqref{ss1} being replaced by $F(u^k)-F(u)-V^k(u^k-u)=O(\|u^k-u\|^2)$.
\end{definition}
To solve the nonlinear equation \eqref{leqt}, the semismooth Newton method, as an instance of nonsmooth (generalized) Newton methods, requires a surrogate in ${\mathds L}(\cU,\cU)$ for the Jacobian used in the classic Newton's method.
For this purpose, one can prescribe a multi-valued mapping $T_F: {O}\rightrightarrows{\mathds L}(\cU,\cU)$ and make the following assumption.
\begin{assumption}
\label{asstf}
The mapping $T_F: {O}\rightrightarrows{\mathds L}(\cU,\cU)$ satisfies
\begin{itemize}
\item[\rm (a)] for any $u\in O$, $T_F(u)$ is a nonempty and compact subset of ${\mathds L}(\cU,\cU)$;
\item[\rm (b)] the mapping $T_F$ is outer semicontinuous\footnote{The definition of outer semicontinuous can be found in Rockafellar \cite[Definition 5.4]{va}.};
\item[\rm (c)] $\partial_B F(u)\Delta u\subset T_F(u)\Delta u\subset\partial_C F(u)\Delta u$,\ \ $\forall\, u\in O,\, \Delta u\in\cU$.
    \end{itemize}
Here, $\partial_C$ is similarly defined as in \cite[Proposition 2.6.2]{clarke}, i.e., by decomposing $\cU=\prod_{i=1}^{n}\cU_i$ with each $\cU_i$ being a finite-dimensional Hilbert space, one can write $F(u)=\prod_{i=1}^{n}F_i(u),\,\forall u\in O$, so that one can define
$$
\partial_C F(u):=\prod_{i=1}^n\partial F_i(u) \quad\mbox{with}\quad
\partial F_i(u):=\pi_{\cU_i}(\partial F(u)),
 \quad\forall u\in O.
$$

\end{assumption}

Based on the mapping $T_F$ defined above, the semismooth Newton method for solving \eqref{leqt} can be stated as follows:
Let $u^0\in O$ be the given initial point.
For $k=1,\ldots$,
\begin{itemize}
\item[{\bf 1.}]
choose $V^{k}\in T_F(u^k)$;
\item[{\bf 2.}]
compute the generalized Newton direction $\Delta u^k=\Delta u$ via solving the linear system $F(u^k)+V^{k} \Delta u=0$;
\item[{\bf 3.}]
set $u^{k+1}:=u^k+\Delta u^k$.
\end{itemize}

\begin{remark}
The notion of semismoothness was first introduced in Mifflin \cite{mifflin} for functionals, and was later extended in Qi and Sun \cite{qi93} to vector-valued functions. Here we adopt the definition for semismoothness from Sun and Sun \cite[Definition 3.6]{sunsun}, which was derived based on Pang and Qi \cite[Proposition 1]{pangqi}.
The semismooth Newton method was first proposed in Kummer \cite{kummer1,kummer2}, and was further extended by Qi and Sun \cite{qi93}, Qi \cite{qi932}, and Sun and Han \cite{sun-newton}. The following theorem gives the convergence properties of the above method.
\end{remark}

\begin{theorem}[{\cite[Theorem 3.2]{qi93}}]
\label{newton}
Let $u^*\in\cU$ be a solution to the nonlinear equation \eqref{leqt}. Suppose that Assumption \ref{asstf} holds and $T_F(u^*)$ is nonsingular, i.e., all $V\in T_F(u^*)$ are nonsingular, and $F$ is semismooth at $u^*$.
Then, there exists a positive number $\varepsilon>0$ such that for any $u^0\in {\mathds B}_{\varepsilon}(u^*)$, the sequence $\{u^k\}$ generated by the semismooth Newton method is well-defined and converges to $u^*$ with a superlinear rate, i.e., for all sufficiently large $k$, it holds that
$$\|u^{k+1}-u^*\|=o(\|u^k-u^*\|).$$
Furthermore, if F is strongly semismooth at $u^*$, then the convergence rate is Q-quadratic, i.e., for all sufficiently large $k$, it holds that $\|u^{k+1}-u^*\|=O(\|u^k-u^*\|^2)$.
\end{theorem}

As can be observed from the proof of \cite[Theorem 3.2]{qi93},
the nonsingularity of $T_{F}(u^*)$ implies that $T_{F}(\cdot)$ is uniformly bounded and nonsingular in a neighborhood of $u^*$.
Meanwhile, the semismoothness of $F$ guarantees that, for any sequence $\{u^{k}\}$ converging to $u^*$ with $k$ being sufficiently large,
\begin{equation}
\label{smallo}
F(u^{k})-F(u^*)-V^{k}(u^{k}-u^*)=o(\|u^{k}-u^*\|), \quad\forall \ V^{k}\in T_{F}(u^{k}).
\end{equation}
In fact, from the proof of \cite[Theorem 3.2]{qi93}, one also can see that, even if the mapping $T_{F}$ fails to satisfy Assumption \ref{asstf}(c), the semismooth Newton method is well-defined and convergent if, additionally, \eqref{smallo} holds.
Consequently a much broader algorithmic framework of nonsmooth (generalized) Newton methods emerges beyond the semismooth newton method.
This observation constitutes the basis for designing the second-order method of multipliers in the subsequent section.
For the convenience of further discussions, we summarize the facts mentioned above as the following proposition.

\begin{proposition}
\label{specialnewton}
Let $u^*\in\cU$ be a solution to the nonlinear equation \eqref{leqt}. Suppose that Assumption \ref{asstf}(a) and Assumption \ref{asstf}(b) hold, and \eqref{smallo} holds for any sequence $\{u^{k}\}$ converging to $u^*$.
If $T_F(u^*)$ is nonsingular and $F$ is semismooth at $u^*$, then there exists a positive number $\varepsilon>0$ such that for any $u^0\in {\mathds B}_{\varepsilon}(u^*)$, the sequence $\{u^k\}$ generated by the semismooth Newton method is well-defined and converges to $u^*$ with a superlinear rate.\end{proposition}
\begin{proof}
By using \eqref{smallo} instead of Assumption \ref{asstf}(c) and repeating the proof of \cite[Theorem 3.2]{qi93},
one can get the result.
\end{proof}

\section{A second-order method of multipliers}
\label{sec3}
In this section, we propose a second-order method of multipliers for solving problem \eqref{op}.
This method is essentially a specially designed nonsmooth (generalized) Newton method, which is based on Proposition \ref{specialnewton}.
Throughout this section, we make the blanket assumption that the projection operator $\Pi_{K}(\cdot)$ is \emph{strongly semismooth}, as this will be fulfilled for all the three classes of problems that will be studied in the forthcoming three sections.

The Lagrangian function of problem \eqref{op} is defined by
\[
\label{laf}
\cL_0(x, \lambda, \mu):=f(x)-\langle\lambda, h(x)\rangle-\langle\mu,g(x)\rangle,
\quad
\forall(x,\lambda,\mu)\in\cX\times\Re^{m}\times\cZ.
\]
A vector $x^*\in\cX$ is called as a stationary point of problem \eqref{op} if there exists a vector $(\lambda^*,\mu^*)\in\Re^{m}\times\cZ$ such that $(x^*,\lambda^*,\mu^*)$ is a solution to the following Karush-Kuhn-Tucker (KKT) system:
\begin{equation}
\label{kkt_2nd}
\left\{
\begin{array}{l}
\nabla_x \cL_0(x,\lambda,\mu)=0, \\[2mm]
h(x)=0, \\[2mm]
g(x)\in K,\ \ \mu\in K,\ \ \langle\mu,g(x)\rangle=0.
\end{array}
\right.
\end{equation}
In this case, the vector $(\lambda^*,\mu^*)$ is called a Lagrange multiplier at $x^*$.
We denote the set of all the Lagrangian multipliers at $x\in\cX$ by $\cM(x)$, which is an empty set if $x$ fails to be a stationary point.

Let $c>0$ be a given constant and $x^*\in\cX$ be a given stationary point of problem \eqref{op} with the Lagrangian multipliers $(\lambda^*,\mu^*)\in\cM(x^*)$. We define the mapping
\begin{equation}
\label{def:A}
\begin{array}{r}
\cA_{(c,\lambda^*,\mu^*)}(W):=\nabla_{xx}^2\cL_0(x^*,\lambda^*,\mu^*)+c\nabla h(x^*)\cJ h(x^*)+c\nabla g(x^*)W\cJ g(x^*),
\quad
\\[2mm]
\forall\,W\in{\mathds L}(\cZ,\cZ).
\end{array}
\end{equation}
We make the following abstract assumption for the given stationary point $x^*\in\cX$.

\begin{assumption}
\label{assg}
{\bf (i)} $(\lambda^*,\mu^*)\in\cM(x^*)$ is the unique Lagrangian multiplier;

\noindent
{\bf (ii)} there exist two positive constants $\bar c$ and $\eta$ such that for any $c\ge \bar c$,
$$
\langle x,\cA_{(c,\lambda^*,\mu^*)}(W)\, x\rangle\ge\eta\| x\|^{2},
\quad
\forall\, W\in\partial_B\Pi_{K}(\mu^*-cg(x^*)),
\
\forall\, x\in \cX,
$$
where the linear operator ${\cA}_{(c,\lambda^*,\mu^*)}(W)$ is defined by \eqref{def:A}.
\end{assumption}
The following result was given in Sun et al. \cite[Proposition 1]{sun07}.

\begin{proposition}
\label{summary}
Suppose that Assumption \ref{assg} holds and $c\ge\bar c$ is fixed. Then, there exist two constants $\varepsilon>0$ and $\delta_0> 0$, both depending on $c$, and a locally Lipschitz continuous function $\bfx_c(\cdot)$ defined on ${\mathds B}_{\delta_0}(\lambda^*,\mu^*)$ by
\[
\label{xcy}
\bfx_c(\lambda,\mu):=\argmin_{x\in {\mathds B}_\varepsilon(x^*)}\cL_c(x,\lambda,\mu),
\]
such that:
\begin{itemize}[topsep=1pt,itemsep=-.6ex,partopsep=1ex,parsep=1ex,leftmargin=4.5ex]
\item[\rm (a)] $\bfx_c(\cdot)$ is semismooth at any point in ${\mathds B}_{\delta_0}(\lambda^*,\mu^*)$;
\item[\rm (b)] for any $x\in {\mathds B}_{\varepsilon}(x^*)$ and $(\lambda,\mu)\in {\mathds B}_{\delta_0}(\lambda^*,\mu^*)$, it holds that every element in $\pi_x\partial_B(\nabla_x\cL_c)(x,\lambda,\mu)$ is positive definite;
\item[\rm (c)] For any $(\lambda,\mu)\in {\mathds B}_{\delta_0}(\lambda^*,\mu^*)$, $\bfx_c(\lambda,\mu)$ is the unique optimal solution to $$\min_{x\in {\mathds B_\varepsilon(x^*)}}\cL_c(x,\lambda,\mu).$$
\end{itemize}
\end{proposition}

If Assumption \ref{assg} holds and $c\ge \bar c$ is fixed, we can also fix the two parameters $\varepsilon>0$ and $\delta_0>0$ such that the conclusions (a)--(c) in Proposition \ref{summary} hold.
In this case, we can define the real-valued function
\begin{equation}
\label{vartheta}
\vartheta_c(\lambda,\mu):=\min_{x\in {\mathds B}_{\varepsilon}(x^*)}\cL_c(x,\lambda,\mu),\quad \forall (\lambda,\mu)\in{\mathds B}_{\delta_0}(\lambda^*,\mu^*).
\end{equation}
Obviously, $\vartheta_c(\cdot)$ is concave on its domain and
$$\vartheta_c(\lambda,\mu)=\cL_c(\bfx_c(\lambda,\mu),\lambda,\mu),\ \forall (\lambda,\mu)\in {\mathds B}_{\delta_0}(\lambda^*,\mu^*).$$
Moreover, the following result can be obtained from Proposition \ref{summary}, and one may refer to Sun et al. \cite[Proposition 2]{sun07} for the detailed proof.
\begin{lemma}
\label{lemma:nal}
Suppose that Assumption \ref{assg} holds and $c\ge \bar c$ is fixed.
Let the two constants $\varepsilon>0$ and $\delta_0>0$ be given by Proposition \ref{summary}.
Then, the concave function $\vartheta_c(\cdot)$ defined by \eqref{vartheta} is continuously differentiable on ${\mathds B}_{\delta_0}(\lambda^*,\mu^*)$ with its gradient given by
\[
\label{nablavartheta}
\nabla\vartheta_c(\lambda,\mu)=\left(\begin{matrix}
-h(\bfx_c(\lambda,\mu))
\\[1mm]
\displaystyle
\frac{1}{c}\,\Pi_K\big(\mu-cg(\bfx_c(\lambda,\mu))\big)-\frac{\mu}{c}
\end{matrix}\right).
\]
Moreover, $\nabla\vartheta_c(\cdot)$ is semismooth on ${\mathds B}_{\delta_0}(\lambda^*,\mu^*)$.
\end{lemma}

Now we start to propose our second-order method of multipliers for solving problem \eqref{op}. Suppose that we have obtained a certain $(\lambda,\mu)\in {\mathds B}_{\delta_0}(\lambda^*,\mu^*)$, which we denote by $(\lambda^k,\mu^k)$ with the index $k\ge 0$ for convenience. Then, the procedure in \eqref{update} to get $(\lambda^{k+1},\mu^{k+1})$ with $x^{k+1}:=\bfx_c(\lambda^k,\mu^k)$ can be reformulated as
$$
(\lambda^{k+1},\mu^{k+1})
=(\lambda^{k},\mu^{k})
+c\nabla\vartheta_c(\lambda^k,\mu^k),
$$
which can be interpreted as a (first-order) gradient ascent step at $(\lambda^k,\mu^k)$ for the purpose of maximizing $\vartheta_c$.
Thus, it is natural to consider if one can introduce a second-order update rule for the multipliers based on this (first-order) gradient information.
Obviously, the classical Newton method cannot be applied directly, due to the nonsmooth property induced by the projection $\Pi_K(\cdot)$.
However, even though the function $\nabla\vartheta_c$ is semismooth on ${\mathds B}_{\delta_0}(\lambda^*,\mu^*)$ (according to Lemma \ref{lemma:nal}), the semismooth Newton method introduced in Section \ref{ssn} is still not directly applicable, inasmuch as the fact that $\nabla\vartheta_c$ is a composite function and it is highly difficult to find a mapping $T_{\vartheta_c}:\Re^m\times\cZ\rightrightarrows{\mathds L}(\Re^m\times\cZ,\Re^m\times\cZ)$ such that
\[
\label{pbvt}
\begin{array}{rl}
\partial_B (\nabla\vartheta_c)(\lambda,\mu)y
\subset T_{\vartheta_c}(\lambda,\mu)y
\subset \partial_C(\nabla\vartheta_c)(\lambda,\mu)y,\qquad\qquad
\\[2mm]
\qquad\qquad
 \forall\, (\lambda,\mu)\in{\mathds B}_{\delta_0}(\lambda^*,\mu^*),\ \forall\,y\in\Re^m\times\cZ.
\end{array}
\]
Fortunately, thanks to Proposition \ref{specialnewton}, even if \eqref{pbvt} fails, one can still hope to design a second-order method to address this issue.
Next, we will design a specific nonsmooth Newton method to ensure a second-order method of multipliers for problem \eqref{op}.

For this purpose, we suppose that Assumption \ref{assg} holds and $c\ge\bar c$ is fixed, so that we can define the functions
\[
\label{multiplus}
\left\{
\begin{array}{ll}
{\lambda}_c(\lambda,\mu):=\lambda-ch(\bfx_c(\lambda,\mu)),
\\[2mm]
\mu_c(\lambda,\mu):=\Pi_K\left(\mu-cg(\bfx_c(\lambda,\mu))\right),
\end{array}\right.
\quad
\forall (\lambda,\mu)\in {\mathds B}_{\delta_0}(\lambda^*,\mu^*),
\]
and
\[
\label{a2}
\begin{array}{r}
\displaystyle
\cA_c(\lambda,\mu,W):=\nabla_{xx}^2\cL_0(\bfx_c(\lambda,\mu),\lambda_c(\lambda,\mu),\mu_c(\lambda,\mu))
+c\nabla h(\bfx_c(\lambda,\mu))\cJ h(\bfx_c(\lambda,\mu))
\qquad
\\[2mm]
+c\nabla g(\bfx_c(\lambda,\mu))W\cJ g(\bfx_c(\lambda,\mu)),
\qquad
\forall\,(\lambda,\mu)\in {\mathds B}_{\delta_0}(\lambda^*,\mu^*),\
\forall\,W\in{\mathds L}(\cZ,\cZ),
\end{array}
\]
where both the constant $\delta_0$ and the function $\bfx_c(
\cdot)$ are specified by Proposition \ref{summary}.
In this case, we can also define the following set-valued mapping (from $\Re^m\times\cZ$ to ${\mathds L}(\Re^m\times\cZ,\Re^m\times\cZ)$)
\[
\label{v}
\begin{array}{l}
{\mathds V}_c(\lambda,\mu)
:=
\left\{
-
\begin{pmatrix}
\cJ h(\bfx_c(\lambda,\mu))\\[2mm]
W\cJ g(\bfx_c(\lambda,\mu))
\end{pmatrix}
(\cA_c(\lambda,\mu,W))^{-1}
\Big(\nabla h(\bfx_c(\lambda,\mu)),\nabla g(\bfx_c(\lambda,\mu))W \Big)
\right.
\quad
\\[6mm]
\hfill
\left.
-\begin{pmatrix}
0&0\\[2mm]
0& \frac{1}{c}(\cI_{\cZ}-W)
\end{pmatrix}
\mid
W\in\partial_B\Pi_K\Big(\mu-cg\big(\bfx_c(\lambda,\mu)\big)\Big)
\right\},
\end{array}
\]
where $\cI_\cZ$ denotes the identity operator in $\cZ$.
According to \eqref{kkt_2nd} and \eqref{def:A}, one has that
\[
\label{alimit}
\cA_c(\lambda^*,\mu^*,W)=\cA_{(c,\lambda^*,\mu^*)}(W),
\quad\forall\, W\in\partial_B\Pi_K(\mu^*-cg(x^*)).
\]
The following result on the properties of the mapping ${\mathds V}_c$ defined above is crucial to the forthcoming algorithmic design. Its proof is given in Appendix \ref{apppp2}, which is mainly based on the work of \cite{chenlalm}.

\begin{proposition}
\label{ppvy}
Suppose that Assumption \ref{assg} holds and $c\ge \bar c$. Then, there exists a constant $\delta_1>0$ ($\delta_1\le \delta_0$) such that the mapping ${\mathds V}_c$ in \eqref{v} is well-defined, compact-valued and outer semicontinuous on ${\mathds B}_{\delta_1}(\lambda^*,\mu^*)$. Moreover, it holds that $$\partial_B(\nabla\vartheta_c)(\lambda,\mu)
\subseteq
{\mathds V}_{c}(\lambda,\mu),\quad
\forall(\lambda,\mu)\in{\mathds B}_{\delta_1}(\lambda^*,\mu^*).$$
\end{proposition}

Based on the above discussions, we are able to present the following second-order method of multipliers.
\begin{algorithm}[ht]
\caption{A second-order method of multipliers for problem \eqref{op}.}
\label{algo-2nd}
\leftline{Given $c>0$ and $(\lambda^0,\mu^0)\in \Re^m\times\cZ$. For $k=0,1,\ldots$,}
\begin{itemize}[topsep=1pt,itemsep=0ex,partopsep=1ex,parsep=1ex,leftmargin=4ex]
\item[\bf 1.] compute $\displaystyle x^{k+1}: =\bfx_c(\lambda^{k},\mu^k)=\argmin_{x\in {\mathds B}_\varepsilon(x^*)}\cL_c(x,\lambda^{k},\mu^k)$;
\item [\bf 2.] choose a linear operator $V^k\in{\mathds V}_c(\lambda^k,\mu^k)$ and solve the linear system
$\nabla\vartheta_c(\lambda^k,\mu^k)+V^k\Delta y=0$ to obtain
$\Delta y^{k+1}\in\Re^m\times\cZ$, where $\nabla\vartheta_c$ is given by \eqref{nablavartheta} and ${\mathds V}_c(\lambda^k,\mu^k)$ is given by \eqref{v};
\item [\bf 3.] compute
$
(\lambda^{k+1},\mu^{k+1}):=(\lambda^k,\mu^k)+\Delta y^{k+1}.
$
\end{itemize}
\end{algorithm}

\begin{remark}
{\rm (a)} Algorithm \ref{algo-2nd} is essentially an application of a nonsmooth Newton method for (locally) maximizing the function $\vartheta_{c}$, or equivalently, for solving the nonlinear equation $\nabla\vartheta_{c}(\lambda,\mu)=0$.
In the second step of Algorithm \ref{algo-2nd}, the linear operator $W\in\partial_{B}\Pi_{K}(\mu^{k}-cg(x^{k+1}))$ for computing $V^k$ can be explicitly obtained for the NLP, NLSOCP and NLSDP problems.
Consequently, one can explicitly compute $\cA_{c}(\lambda^{k},\mu^k,W^{k})$ via \eqref{a2} and the linear operator $V^{k}$ can be calculated via \eqref{v}.

{\rm (b)}
Generally, it is hard to check if ${\mathds V}_{c}(\lambda^{k},\mu^k)\subseteq\partial_{C}(\nabla\vartheta_c)(\lambda^k,\mu^k)$.
Consequently, Algorithm \ref{algo-2nd} could \emph{not} be attributed as the classic semismooth Newton method introduced in Section \ref{ssn}.
We should mention that, it is still uncertain if one can explicitly find a linear operator $V^{k}$ such that $V^{k}\in\partial_{C}(\nabla\vartheta_c)(\lambda^k,\mu^k)$ or $V^{k}\in\partial_{B}(\nabla\vartheta_c)(\lambda^k,\mu^k)$.
However, as will be shown later, $V^{k}\in{\mathds V}_{c}(\lambda^k,\mu^k)$ is sufficient for our purpose.

\end{remark}

To analyze the convergence properties of Algorithm \ref{algo-2nd}, we make the following abstract assumption as a supplementary to Assumption \ref{assg}.
\begin{assumption}
\label{ass:pd}
For any $c\ge\bar c$, any linear operator $V\in{\mathds V}_c(\lambda^*,\mu^*)$ satisfies $V\prec 0$.
\end{assumption}

The following result is essential to the convergence analysis of Algorithm \ref{algo-2nd}. The proof of this theorem is given in Appendix \ref{ap3}.

\begin{theorem}
\label{th31}
Suppose that Assumptions \ref{assg} and \ref{ass:pd} hold, and $c\ge \bar c$. Let $\{(\lambda^{k},\mu^{k})\}$ be an arbitrary sequence that converging to $(\lambda^*,\mu^*)$. Then, for all the sufficiently large $k$, the set
${\mathds V}_c(\lambda^{k},\mu^{k})$ is nonempty and compact, and it holds that
\[
\label{gt}
\begin{array}{r}
\nabla\vartheta_c(\lambda^{k},\mu^{k})-\nabla\vartheta_c(\lambda^*,\mu^*)
-V_{k}(\lambda^{k}-\lambda^*,\mu^{k}-\mu^*)
=o\left(\|(\lambda^{k}-\lambda^*,\mu^{k}-\mu^*)\|\right),\
\\[2mm]
 \forall\,
V^{k}\in{\mathds V}_c(\lambda^{k},\mu^k).
\end{array}
\]
\end{theorem}

Now we establish the local convergence properties of Algorithm 1.
\begin{theorem}
\label{convergence}
Suppose that Assumptions \ref{assg} and \ref{ass:pd} hold.
If $c>0$ is sufficiently large, then there exists a constant $\delta>0$ such that, if
$(\lambda^0,\mu^0)\in{\mathds B_{\delta}}(\lambda^*,\mu^*)$, the whole sequence $\{(\lambda^{k},\mu^k)\}$ generated by Algorithm \ref{algo-2nd} is well-defined and converges to $(\lambda^*,\mu^*)$ with a superlinear convergence rate.
\end{theorem}
\begin{proof}
According to Lemma \ref{lemma:nal}, the function $\nabla\vartheta_c$ defined by \eqref{nablavartheta} is semismooth. Meanwhile, the mapping ${\mathds V}_c$ defined in \eqref{v} is locally nonempty, compact-valued and outer semicontinuous in a neighborhood of $(\lambda^*,\mu^*)$.
Then, by Proposition \ref{summary} and Proposition \ref{ppvy} we know that Algorithm \ref{algo-2nd} is well-defined.
Consequently, the theorem follows from Theorem \ref{th31} and Proposition \ref{specialnewton}.
This completes the proof.
\end{proof}

We make the following remark on the abstract assumptions used in Theorem \ref{convergence}.
\begin{remark}
Theorem \ref{convergence} has established the local convergence, as well as the superlinear convergence rate, of the second-order method of multipliers (Algorithm \ref{algo-2nd}).
However, at the first glance, both Assumptions \ref{assg} and \ref{ass:pd} are too abstract to be appreciated.
Therefore, it is of great importance to transfer them to some concrete conditions.
In the following three sections, we show separately that, for the three typical scenarios, i.e., the NLP, the NLSOCP, and the NLSDP, Assumptions \ref{assg} and \ref{ass:pd} are the consequence of a certain constraint qualification, together with a second-order sufficient condition, and the corresponding analysis constitutes the main contribution of this paper.
As can be observed later, these concrete conditions are exactly the prerequisites that were frequently used for establishing the local Q-linear convergence rate of the method of multipliers without requiring strict complementarity.
\end{remark}

Before finishing this section, we would like to make the following remark on Algorithm \ref{algo-2nd}.
\begin{remark}

In the conventional problem settings of NLP, NLSOCP and NLSDP, even the classic augmented Lagrangian method is not necessarily convergent globally.
Moreover, Algorithm \ref{algo-2nd} is based on a specially designed generalized Newton method by regarding the update of multipliers in the augmented Lagrangian method as a first-order step.
As the second-order approaches often require stronger conditions to guarantee even the local convergence property, it is not necessarily globally convergent in general if the first-order counterpart is not.
Consequently, the best possible result for the multiplier sequence generated by Algorithm \ref{algo-2nd} is a local convergence.
Moreover, as the convergence rate study relies on a certain local error bound (as it is not realistic in general to expect a global error bound), the corresponding convergence rate results are also local ones.

\end{remark}
\color{black}

\section{The NLP case}
\label{sec4}
Consider the NLP problem
\begin{equation}
\min_x f(x)\quad\mbox{s.t.}\quad h(x)=0,\quad g(x)\ge 0,
\label{p}
\end{equation}
where
$f:\Re^n\to\Re$, $h:\Re^n\to \Re^{m}$ and $g:\Re^n\to\Re^{p}$ are three twice continuously differentiable functions.
Problem \eqref{p} is an instance of problem \eqref{op} with $\cX\equiv\Re^n$, $\cZ\equiv\Re^{p}$ and $K\equiv\Re_+^{p}:=\{\mu\in\Re^{p}\mid \mu_j\ge 0, j=1,2,\ldots, p\}$.
For convenience,
we define the two index sets ${\mathds I}:=\{1,\ldots,m\}$ and ${\mathds J}:=\{1,\ldots,p\}$.
The following two definitions are well-known \cite{numopt} for NLP problem \eqref{p}.
\begin{definition}[LICQ]
\label{deflicq}
Given a point $x\in\Re^n$ such that $h(x)=0$ and $g(x)\ge 0$.
We say that the linear
independence constraint qualification (LICQ) holds at $x$ if the set of active constraint gradients
$$
\{\nabla h_i(x), \nabla g_j(x)\mid \, i\in{\mathds I},\, j\in{\mathds J}, g_j(x)=0\}
$$
is linearly independent.
\end{definition}

Let $(x^*,\lambda^*,\mu^*)$ be a solution to the KKT system (\ref{kkt_2nd}).
The critical cone at $x^*$ for $(\lambda^*,\mu^*)$
is defined by
$$
\begin{array}{ll}
\cC(x^*,(\lambda^*,\mu^*))
\\[2mm]
\quad
:=\left\{d\in\Re^m \Bigg|
\begin{array}{l}
\langle \nabla h_i(x^*),d\rangle=0,\ \forall i \in{\mathds I},
\\[.5mm]
\langle \nabla g_j(x^*),d\rangle=0,\ \forall j \in{\mathds J}\mbox{ such that } g_j(x^*)=0 \mbox{ and } \mu^*_{j}>0,
\\[.5mm]
\langle \nabla g_j(x^*),d\rangle\ge 0,\ \forall j \in{\mathds J}\mbox{ such that } g_j(x^*)=0 \mbox{ and } \mu^*_{j}=0
\end{array}
\right\}.
\end{array}
$$
\begin{definition}[Second-order sufficient condition for NLP]
Let $x^*$ be a stationary point of problem \eqref{p},
We say the second-order sufficient condition holds at $x^*$ if there exists $(\lambda^*,\mu^*)\in\cM(x^*)$ such that
$$
\langle d,\nabla_{xx}^2\cL_{0}(x^*,\lambda^*,\mu^*)d\rangle >0,\ \forall d\in \cC(x^*, (\lambda^*,\mu^*))\setminus\{0\}.
$$
\end{definition}

The following definition was introduced in Robinson \cite{rob80}.
\begin{definition}[Strong second-order sufficient condition for NLP]
\label{defssoscnlp}
Let $x^*$ be a stationary point of problem \eqref{p},
We say the strong second-order sufficient condition holds at $x^*$ if there exists $(\lambda^*,\mu^*)\in\cM(x^*)$ such that
$$\langle d, \nabla_{xx}^2\cL_{0}(x^*,\lambda^*,\mu^*)d\rangle>0,\ \forall d\in \aff\cC(x^*, (\lambda^*,\mu^*))\setminus\{0\},$$
where $\aff\cC(x^*, (\lambda^*,\mu^*))$ is the affine hull of $\cC(x^*,(\lambda^*,\mu^*))$.
\end{definition}
The following proposition is the main result of Chen and Liao \cite{chenlalm}.
\begin{proposition}
Suppose that both the LICQ and the strong second-order sufficient condition hold at the given stationary point $x^*$ of the NLP problem \eqref{p}.
Then, Assumptions \ref{assg} and \ref{ass:pd} hold at $x^*$ for problem \eqref{p}.
\end{proposition}

\section{The NLSOCP case}
\label{sec5}

In this section, we consider the following NLSOCP problem
\begin{equation}
\min f(x)\quad\mbox{s.t.}\quad h(x)=0,\quad g(x)\in K,
\label{socp}
\end{equation}
where $f: \Re^n\to\Re$, $h: \Re^n\to \Re^{m}$, and $g: \Re^n\to\Re^p$ are three twice continuously differentiable functions, $K:=\cQ_1\times\cdots\times\cQ_s$ with each $\cQ_j$, $j=1,\ldots,s$ being a second-order cone in $\Re^{r_j+1}$, i.e.,
$$
\cQ_j:=\{
\mu_j=(\bar\mu_j, \dot{\mu}_j)\in\Re^{r_j}\times\Re\mid \dot{\mu}_j\ge\|\bar{\mu}_j\|
\}.
$$
Note that $p=\sum_{j=1}^s(r_j+1)$.
Problem \eqref{socp} is an instance of problem \eqref{op} with $\cX\equiv\Re^n$ and $\cZ\equiv\Re^{p}$.
 Moreover, for convenience, we can partition the function $g$ by $g(\cdot)=(g_1(\cdot),\ldots,g_s(\cdot))$ with each $g_j(\cdot)=(\bar{g}_j(\cdot),\dot{g}_j(\cdot))$ such that $\dot{g}_j$ is real-valued.
In the rest part of this section, we first introduce some preliminary results related to NLSOCP problems. After that, we present the main results of this section, and discuss a special case.

\subsection{Preliminaries on NLSOCP}
For each $j=1,\dots,s$, we use $\int\cQ_j$ and $\bdry\cQ_j$ to denote the interior and the boundary of the second-order cone $\cQ_j$, respectively.
According to \cite[Proposition 3.3]{fuk01}, the projection operator to the cone $K$ takes the form of $\Pi_K(\mu)=(\nu_1,\ldots,\nu_s)$, $\forall\mu=(\mu_1,\ldots,\mu_s)\in\Re^{r_1+1}\times\cdots\times\Re^{r_s+1}$ with
$$
\nu_j=\left\{
\begin{array}{ll}
\displaystyle \frac{1}{2}\left(1+\frac{\dot{\mu}_j}{\|\bar{\mu}_j\|}\right)(\bar{\mu}_j,\|\bar{\mu}_j\|), &\mbox{if }\ |\dot{\mu}_j| < \|\bar{\mu}_j\|,
\\[2mm]
\mu_j, &\mbox{if }\ \dot{\mu}_j\ge\|\bar{\mu}_j\|,
\\[2mm]
0, &\mbox{if }\ \dot{\mu}_j\le-\|\bar{\mu}_j\|, \\
\end{array}
\right.
$$
where
$
\mu_j=(\bar{\mu}_j,\dot{\mu}_j)\in\Re^{r_j}\times\Re, \forall j=1,\ldots,s.
$

It has been established in Chen et al. \cite[Proposition 4.3]{chenxd} (also in Hayashi et al. \cite[Proposition 4.5]{haya05}) that the projection operator $\Pi_K(\cdot)$ is strongly semismooth.
Moreover, the Bouligand-subdifferential of this projection operator has been discussed in Pang et al. \cite[Lemma 14]{pangss}, Chen et al. \cite[Lemma 4]{chenjs} and Outrata and Sun \cite[Lemma 1]{sunsoc}.
The following result directly follows \cite[Lemmas 2.5 \& 2.6]{kanzow09}.

\begin{lemma}
\label{parbpik}
For any given point $\mu=(\mu_1,\ldots,\mu_s)\in\Re^{r_1+1}\times\cdots\times\Re^{r_s+1}$, each element $W\in\partial_B\Pi_K(\mu)$ is a block diagonal matrix $W=\diag(W_1,\ldots,W_s)$, with each $W_j\in\Re^{(r_j+1)\times(r_j+1)}$ taking one of the following representations:
\begin{itemize}
\item[\rm (a)]
$
W_j =
\left\{
\begin{array}{ll}
\ds
\frac{1}{2}\begin{pmatrix}
\ds
\left(1+\frac{\dot{\mu}_j}{\|\bar{\mu}_j\|}\right) I_{r_j}
-\frac{\dot{\mu}_j}{\|\bar{ \mu}_j\|}\frac{\bar{\mu}_j\bar{\mu}_j^T}{\|\bar{\mu}_j\|^2}
&\ds\frac{\bar{\mu}_j}{\|\bar{\mu}_j\|}\\[4mm]
\ds\frac{\bar{\mu}_j^T}{\|\bar{\mu}_j\|}&1
\end{pmatrix},
 & \mbox{ if } \ds |\dot{\mu}_j|<\|\bar{\mu}_j\|,\\[2mm]
 \ds
\ds I_{r_j+1}, & \mbox{ if }\ds \dot{\mu}_j> \|\bar{\mu}_j\|,\\[2mm]
\ds0, & \mbox{ if } \ds \dot{\mu}_j< -\|\bar{\mu}_j\|;
\end{array}
\right.
$
\item[\rm (b)]
$
W_j\in\left\{I_{r_j+1},\
I_{r_j+1}+
\frac{1}{2}
\begin{pmatrix}
\ds
-\frac{\bar{\mu}_j\bar{\mu}_j^T}{\|\bar{\mu}_j\|^2}
&\ds\frac{\bar{\mu}_j}{\|\bar{\mu}_j\|}\\[3mm]
\ds\frac{\bar{\mu}_j^T}{\|\bar{\mu}_j\|}&
\ds-1
\end{pmatrix}
\right\},
\
\mbox{ if } \mu_j\neq 0 \mbox{ and } \dot{\mu}_j=\|\bar{\mu}_j\|;
$
\item [\rm (c)]
$
W_j\in\left\{0,\
\frac{1}{2}
\begin{pmatrix}
\ds
\frac{\bar{\mu}_j\bar{\mu}_j^T}{\|\bar{\mu}_j\|^2}
&
\ds\frac{\bar{\mu}_j}{\|\bar{\mu}_j\|}\\[4mm]
\ds
\frac{\bar{\mu}_j^T}{\|\bar{\mu}_j\|}&
\ds 1
\end{pmatrix}
\right\},\
\mbox{ if } \mu_j\neq 0 \mbox{ and } \dot{\mu}_j=-\|\bar{\mu}_j\|;
$

\item [\rm (d)]
$
W_j\in
\left\{
\frac{1}{2}
\begin{pmatrix}
\ds
2\tau I_{r_j}+(1-2\tau)\bar\nu\bar\nu^T &\ds \bar\nu\\[2mm]
\ds \bar\nu^T&\ds 1
\end{pmatrix}
\Bigl\vert\
\bar\nu\in\Re^{r_j},\ \|\bar\nu\|=1,\ \tau\in[0,1]
\right\} \cup \{0\} \cup\{I_{r_j+1}\},\
\mbox{ if } \mu_j=0.
$
\end{itemize}
\end{lemma}
According to \cite[Lemma 25]{bonnans}, the tangent cone of $K$ at $\mu=(\mu_1,\ldots,\mu_s)\in K$, with each $\mu_j\in\Re^{r_j+1}$, $j=1,\ldots,s$, is given by $\cT_{K}(\mu)=\cT_{\cQ_1}(\mu_1)\times\cdots\times\cT_{\cQ_s}(\mu_s)$, where
$$
\cT_{\cQ_j}(\mu_j):=
\left\{
\begin{array}{ll}
\Re^{r_j+1}, & \mbox{ if } \dot{\mu}_j>\bar{\mu}_j, \\[2mm]
\cQ_j, & \mbox{ if } \mu_j=0,\\[1mm]
\{\nu_j\in\Re^{r_j+1}\mid \langle \bar \nu_j,\bar{\mu}_j\rangle-\dot{\nu_j}\dot{\mu}_j\le 0\},\ &\mbox{ otherwise}.
\end{array}
\right.
$$
Consequently, the corresponding lineality space of $\cT_{K}(\mu)$ takes the following form
\[
\label{lintangzsoc}
{\lin}(\cT_{K}(\mu))=\Theta_1\times\cdots\times\Theta_s
\]
with
$$\Theta_j=
\left\{
\begin{array}{ll}
\Re^{r_j+1}, & \mbox{ if } \dot{\mu}_j>\bar{\mu}_j, \\
\{0\}, & \mbox{ if } \mu_j=0,\\
\{\nu_j\in\Re^{r_j+1}\mid \langle \bar \nu_j,\bar{\mu}_j\rangle-\dot{\nu_j}\dot{\mu}_j= 0\},\ &\mbox{ otherwise}.
\end{array}
\right.
$$
The following definition of constraint nondegeneracy\footnote{The name of ``constraint nondegeneracy'' was coined in Robinson in \cite{rob03}. For NLP problems, constraint nondegeneracy is exactly the LICQ, c.f. Robinson \cite{robmp} and Shapiro \cite{sha03}.} for NLSOCP problems follows from \cite[Section 2]{bon98}.

\begin{definition}[Constraint nondegeneracy]
We say that the constraint nondegeneracy holds at a feasible point $x$ to the NLSOCP problem \eqref{socp} if
\[
\label{constraintsoc}
\begin{pmatrix}
\cJ h(x)\\[1mm]
\cJ g(x)
\end{pmatrix}\Re^n
+\begin{pmatrix}
\{0\}\\[1mm]
\lin\big(
\cT_{K}(g(x))
\big)
\end{pmatrix}
=\begin{pmatrix}
\Re^{m}\\[1mm]
\Re^{p}
\end{pmatrix}.
\]
\end{definition}

Let $x^*$ be a stationary point of the NLSOCP problem \eqref{socp} with $(\lambda^*,\mu^*)\in\cM(x^*)$.
The critical cone $\cC(x^*)$ of NLSOCP at $x^*$ is defined, e.g., in \cite{pertub,bonnans}, by
$$
\cC(x^*):= \left\{d\in \Re^n\mid \cJ h(x^*)d=0,\ \cJ g(x^*)d \in \cT_{K}(g(x^*))\cap(\mu^*)^\bot\right\},
$$
where $(\mu^*)^\bot$ is the orthogonal complement of $\mu^*$ in $\Re^p$.
The explicit formulation of $\cC(x^*)$ can be found in \cite[Corollary 26]{bonnans}.
Moreover, as in \cite[(47)]{bonnans}, for any $d\in {\aff}\, \cC(x^*)$, one has that $\cJ h(x^*)d=0$ and for all $j=1,\ldots,s$,
$$
\left\{
\begin{array}{ll}
\cJ g_j(x^*)d=0, &\mbox{ if } \mu^*_j\in{\int\,}\cQ_j,
\\[1mm]
\cJ g_j(x^*)d\in{\aff}\{(-\bar{\mu}_j, \dot{\mu}_j)\}, &\mbox{ if } \mu^*_j\in\bdry\cQ_j\backslash\{0\} \mbox{ and } g_j(x^*)=0,
\\[1mm]
\langle \cJ g_j(x^*)d, \mu_j^*\rangle =0, &\mbox{ if } \mu^*_j,g_j(x^*) \in\bdry\cQ_j\backslash\{0\}.
\end{array}
\right.
$$
Hence, one has that
\[
\label{affrel}
\aff\cC(x^*)=\left\{d\in \Re^n\mid \cJ h(x^*)d=0,\ \cJ g(x^*)d \in \aff\big(\cT_{K}(g(x^*))\cap(\mu^*)^\bot\big)\right\}.
\]
The following definition of the strong second-order sufficient condition for the NLSOCP problem comes from Bonnans and Ram\'irez \cite{bonnans}.
\begin{definition}[Strong second-order sufficient condition for NLSOCP]
Let $x^*$ be a stationary point of the NLSOCP problem \eqref{socp}.
We say that the strong second-order sufficient condition holds at $x^*$ if there exists a vector $(\lambda^*,\mu^*)\in\cM(x^*)$ such that
\[
\label{soscsoc}
\left\langle d,\left(\cJ^2_{xx}\cL_0(x^*,\lambda^*,\mu^*)
+\cH(x^*,\mu^*)\right)d
\right\rangle
>0,
\quad
\forall d\in {\aff}(C(x^*))\backslash\{0\},
\]
where $\cH(x^*,\mu^*)=\sum_{j=1}^s\cH_j(x^*,\mu^*)$ is a matrix with each $\cH_j(x^*,\mu^*)\in\Re^{n\times n}$, $j=1,\ldots,s$ defined by
$$
\begin{array}{l}
\cH_j(x^*,\mu^*)
:=
\left\{
\begin{array}{ll}
\ds-\frac{\dot{\mu}_j}{\dot{g}_j(x^*)}\nabla g_j(x^*)
\begin{pmatrix}
-I_{r_j} &0\\
0& 1
\end{pmatrix}
\cJ g_j(x^*),
&\mbox{if } g_j(x^*)\in\bdry\cQ_j\backslash\{0\},
\\[4mm]
0,&\mbox{ otherwise}.
\end{array}
\right.
\end{array}
$$
\end{definition}

\subsection{Main results}
Based on the previous subsection, we are ready to establish the main result of this section.
\begin{theorem}
Suppose that both the constraint nondegeneracy \eqref{constraintsoc} and the strong second-order sufficient condition \eqref{soscsoc} hold for problem \eqref{socp} at $x^*$ with $(\lambda^*,\mu^*)\in\cM(x^*)$.
Then, Assumptions \ref{assg} and \ref{ass:pd} hold.
\end{theorem}
\begin{proof}
Assumption \ref{assg} directly follows from \cite[Proposition 17]{bonnans} and \cite[Proposition 10]{liuzhang1}.
In the following, we show that Assumption \ref{ass:pd} holds.

Let $c\ge\bar c$ be fixed.
Denote $z=(z_1,\ldots,z_s):=\mu^*-cg(x^*)$ with each $z_j:=\mu_j-cg_j(x^*)$, and $g_j(\cdot)=(\bar g_j(\cdot),\dot g_j(\cdot))$, $\forall j=1,\ldots,s$.
Since Assumption \ref{assg} holds, we know that for any $W\in\partial_B\Pi_{K}(z)$, the linear operator $\cA_{(c,\lambda^*,\mu^*)}(W)$ defined in \eqref{def:A} is positive definite, so that by \eqref{alimit} we know that ${\mathds V}_c(\lambda^*,\mu^*)$ defined by \eqref{v} is nonempty.
Suppose that there exists a certain $W\in\partial_B\Pi_{K}(z)$ such that the corresponding $V\in{\mathds V}_c(y^*)$ is not  {negative} definite.
From \eqref{v} we know that $V\preceq 0$.
Therefore, what we actually suppose is that $V$ is singular, i.e., there exists two vectors $\Delta\lambda\in\Re^{m}$ and $\Delta\mu\in\Re^{p}$ such that
\[
\label{contrac}
\Delta y:=(\Delta\lambda,\Delta\mu)\neq 0\quad\mbox{but}\quad\langle \Delta y, V \Delta y\rangle=0.
\]
Note that
\[
\label{vdetv}
\begin{array}{ll}
-\left\langle \Delta y, V(\Delta y)\right\rangle
=
&
\ds \big\langle
\nabla h(x^*)(\Delta\lambda)+\nabla g(x^*) W(\Delta\mu),
\\[1mm]
&\quad
(\cA_{(c,\lambda^*,\mu^*)}(W))^{-1}
\big(\nabla h(x^*)(\Delta\lambda)+\nabla g(x^*) W(\Delta\mu)\big) \big\rangle
\\[2mm]
&+\frac{1}{c}
\left
\langle \Delta\mu,
(I_{p}-W)(\Delta\mu)
\right
\rangle.
\end{array}
\]
Meanwhile, from Lemma \ref{parbpik} we know that $W=\diag(W_1,\ldots,W_s)$ with each $W_j\in\Re^{(r_j+1)\times(r_j+1)}$, $\forall j=1,\ldots,s$.
Moreover, it is easy to deduce from \cite[Proposition 1]{meng} that $I_{r_j}\succeq W_j\succeq 0$.
Hence, by using the fact that $\cA_{(c,\lambda^*,\mu^*)}(W)\succ 0$, one has from \eqref{contrac} and \eqref{vdetv} that
\[
\label{eq001soc}
\nabla h(x^*)(\Delta\lambda)+\nabla g(x^*) W(\Delta\mu)=0,
\]
and
\[
\label{eq002soc}
\langle (\Delta\mu)_j,
(I_{r_j+1}- W_j)(\Delta\mu)_j
\rangle
=0,\quad\forall j=1,\ldots,s.
\]
Since the constraint nondegeneracy condition \eqref{constraintsoc} holds at $x^*$,
there exist a vector $d\in \Re^n$ and a vector $\xi=(\xi_1,\ldots,\xi_s)\in{\lin}\big(
\cT_{K}(g(x^*))
\big)$ with $\xi_j=(\bar{\xi}_j,\dot{\xi}_j)\in\Re^{r_j}\times\Re$, $\forall j=1,\ldots,s$, such that
$$
\cJ h(x^*)d=\Delta\lambda
\quad
\mbox{and}\quad
\cJ g(x^*)d+\xi=\Delta\mu.
$$
Therefore, one can get from \eqref{eq001soc} and \eqref{eq002soc} that
\[
\label{deltaplus}
\begin{array}{l}
\|\Delta\lambda\|^2+\|\Delta\mu\|^2
=\langle \Delta\lambda,\Delta\lambda\rangle+\langle\Delta\mu, \Delta\mu\rangle
\\[2mm]
=\langle \Delta\lambda,\Delta\lambda\rangle
+\langle W\Delta\mu, \Delta\mu\rangle
+\langle (I_{p}-W)\Delta\mu, \Delta\mu\rangle
\\[2mm]
=\langle \Delta\lambda, \cJ h(x^*)d\rangle
+\langle W\Delta\mu, \cJ g(x^*)d+\xi\rangle
+\langle (I_{p}- W)\Delta\mu, \Delta\mu\rangle
\\[2mm]
=\langle \nabla h(x^*)(\Delta\lambda)+\nabla g(x^*) W(\Delta\mu), d\rangle
+\langle W\Delta\mu, \xi\rangle
+\langle (I_{p}- W)\Delta\mu, \Delta\mu\rangle
\\[2mm]
=\langle W\Delta\mu, \xi\rangle.
\end{array}
\]
Since $\xi\in{\lin}\big(\cT_{K}(g(x^*))\big)$, one has from \eqref{lintangzsoc} that for any $j=1,\ldots,s$,
\[
\label{gsrelation}
\left\{
\begin{array}{ll}
\xi_j\in\Re^{r_j+1}, & \mbox{ if } g_j(x^*)\in {\int\, }\cQ_j ,
\\[1mm]
\xi_j=0, & \mbox{ if } g_j(x^*)=0,
\\[1mm]
\langle \bar{\xi}_j,\bar{g}_j(x^*)\rangle-\dot{\xi}_j \dot{g}_j(x^*)= 0, &\mbox{ otherwise.}
\end{array}
\right.
\]
Now, for each $j\in\{1,\ldots,s\}$, we separate our analysis to the following four cases.

{\bf Case 1. }If $g_j(x^*)\in{\int }\, \cQ_j$, one has that $\mu_j^*=0$ so that $-z_j=cg_j(x^*)\in{\int}\,\cQ_j$. Therefore, by Lemma \ref{parbpik} (a) we know that $W_j=0$. Hence, in this case, $\langle W_j(\Delta\mu)_j, \xi_j\rangle=0$.

 {\bf Case 2. }If $g_j(x^*)=0$, one has that $\xi_j=0$. Hence, $\langle W_j(\Delta\mu)_j, \xi_j\rangle=0$ also holds.

{\bf Case 3. }If $g_j(x^*)\in \bdry\, \cQ_j\backslash \{0\}$ and $\mu_j^*=0$, one has that $-z_j\in\bdry \cQ_j\backslash\{0\}$. Hence, in this case, we know from Lemma \ref{parbpik} (c) that
$$
W_j\in\left\{0,
\frac{1}{2}
\begin{pmatrix}
\ds
\frac{\bar{z}_j \bar{z}_j^T}{\|\bar{z}_j\|^2}
&\ \ds\frac{\bar{z}_j}{\|\bar{z}_j\|}
\\[3mm]
\ds\frac{\bar{z}_j^T}{\|\bar{z}_j\|}&\ds 1
\end{pmatrix}
\right\}.
$$
If $W_j=0$, one has that $\langle W_j(\Delta\mu)_j, \xi_j\rangle=0$. Otherwise, one has that
\[
\label{wxi3}
\begin{array}{ll}
W_j\xi_j
&=\frac{1}{2}
\begin{pmatrix}
\ds
\frac{\bar{z}_j \bar{z}_j^T}{\|\bar{z}_j\|^2}
&\ds\frac{\bar{z}_j}{\|\bar{z}_j\|}
\\[3mm]
\ds\frac{\bar{z}_j^T}{\|\bar{z}_j\|}&\ds 1
\end{pmatrix}\xi_j
=
\frac{1}{2}\begin{pmatrix}
\ds
\frac{(\bar{z}_j^T\bar{\xi}_j)\bar{z}_j }{\|\bar{z}_j\|^2}
+\frac{\dot{\xi}_j\bar{z}_j}{\|\bar{z}_j\|}
\\[3mm]
\ds\frac{\bar{z}_j^T\bar{\xi}_j}{\|\bar{z}_j\|}+\dot{\xi}_j
\end{pmatrix}
\\[10mm]
&=\frac{1}{2}
\begin{pmatrix}
\ds\frac{1}{{\|\bar{z}_j\|^2}}
\left(\bar{z}_j^T \bar{\xi}_j
+{\dot{\xi}_j}{\| \bar{z}_j\|}\right) \bar{z}_j
\\[3mm]
\ds\frac{1}{\| \bar{z}_j\|}
\left( \bar{z}_j^T \bar{\xi}_j+\dot{\xi}_j\| \bar{z}_j\|\right)
\end{pmatrix}.
\end{array}
\]
Note that $\mu_j^*=0$, $\dot{z}_j=-\|\bar{z}_j\|$, and $\dot g_j(x^*)>0$. Then, one can see from \eqref{gsrelation} that
$
\bar{\xi}_j^T\bar{z}_j +\dot{\xi}_j\| \bar{z}_j\|
=\bar{\xi}_j^T\bar{z}_j -\dot{\xi}_j\dot{z}_j
=-c\langle \bar{\xi}_j, \bar g_j(x^*)\rangle
+c\dot\xi_j\dot g_j(x^*)=0.
$
Therefore, one can get from \eqref{wxi3} that $\langle W_j(\Delta\mu)_j, \xi_j\rangle=0$.

{\bf Case 4. }If $g_j(x^*)\in \bdry\cQ_j\backslash \{0\}$ and $\mu_j^*\neq 0$, from \cite[Lemma 2.2]{liuzhang2}\footnote{The proof of this Lemma originally comes from \cite{aliz}.} we know that there exists a constant $\sigma_j>0$ such that $\mu_j^*=\sigma_j(-\bar{g}_j(x^*),\dot{g}_j(x^*))$.
Hence, in this case, it holds that
\[
\label{zexp}
z_j=\mu_j^*-cg_j(x^*)=\big(-(\sigma_j+c)\bar{g}_j(x^*),(\sigma_j-c)\dot{g}(x^*)\big).
\]
Since $\sigma_j+c> \sigma_j-c>-(\sigma_j+c)$, it holds that
$-\|\bar{z}_j\|< \dot{z}_j<\|\bar{z}_j\|$.
Therefore, by using Lemma \ref{parbpik} (a), the equation \eqref{zexp}, and the fact that
$\|\bar{g}_j(x^*)\|=\dot{g}_j(x^*)$, one can get that
\[
\label{wjsocp}
\begin{array}{lll}
\ds
W_j &=&\ds \frac{1}{2}\begin{pmatrix}
\ds
\left(1+\frac{\dot{z}_j}{\|\bar{z}_j\|}\right)
I_{r_j}-\frac{\dot{z}_j}{\|\bar{z}_j\|}\frac{\bar{z}_j\bar{z}_j^T}{\|\bar{z}_j\|^2}
&\ds\frac{\bar{z}_j}{\|\bar{z}_j\|}\\[4mm]
\ds\frac{\bar{z}_j^T}{\|\bar{z}_j\|}&\ds 1
\end{pmatrix}
\\[10mm]
&=&\ds\frac{1}{2}\begin{pmatrix}
\ds
\left(1+\frac{\sigma_j-c}{\sigma_j+c}\right)I_{r_j}
-\frac{\sigma_j-c}{\sigma_j+c}
\frac{ \bar{g}_j(x^*)(\bar{g}_j(x^*))^T}{\|\bar{g}_j(x^*)\|^2}
&\ds\quad\frac{-\bar{g}_j(x^*)}{\|\bar{g}_j(x^*)\|}
\\[4mm]
\ds
\frac{(-\bar{g}_j(x^*))^T}{\|\bar{g}_j(x^*)\|}& \ds 1
\end{pmatrix}.
\end{array}
\]
Therefore, it holds that
$$
\begin{array}{ll}
(I_{r_j+1}-W_j)(\Delta\mu)_j
\\[2mm]
=
\frac{1}{2}\begin{pmatrix}
\ds
\frac{2c}{\sigma_j+c}{{(\Delta\bar\mu)}_j}
+\frac{\sigma_j-c}{\sigma_j+c}
\frac{(\bar{g}_j(x^*))^T(\Delta\bar\mu)_j}{\|\bar{g}_j(x^*)\|^2}\bar{g}_j(x^*)
+\frac{(\Delta\dot\mu)_j}{\|\bar{g}_j(x^*)\|}\bar{g}_j(x^*)
\\[4mm]
\ds\frac{(\bar{g}_j(x^*))^T(\Delta\bar\mu)_j}{\|\bar{g}_j(x^*)\|}+(\Delta\dot\mu)_j
\end{pmatrix},
\end{array}
$$
which, together with \eqref{eq002soc}, implies that
$$
\frac{2c}{\sigma_j+c}\|(\Delta\bar\mu)_j\|^2
+\frac{\sigma_j-c}{\sigma_j+c}
\frac{[(\bar g_j(x^*))^T(\Delta\bar\mu)_j]^2}{\|\bar{g}_j(x^*)\|^2}
+2\frac{(\Delta\dot\mu)_j(\bar g_j(x^*))^T(\Delta\bar\mu)_j
}{\|\bar{g}_j(x^*)\|}
+(\Delta\dot\mu)_j^2
=0.
$$
If $(\Delta\bar\mu)_j=0$, the above equation implies that $(\Delta\dot\mu)_j=0$ so that $\langle W_j(\Delta\mu)_j, \xi_j\rangle=0$.
Otherwise, one has that $(\Delta\bar\mu)_j\neq 0$, and the above equation can be equivalently formulated as
$$
\frac{2c}{\sigma_j+c}
+\frac{\sigma_j-c}{\sigma_j+c}
\frac{[(\bar g_j(x^*))^T(\Delta\bar\mu)_j]^2}{\|\bar{g}_j(x^*)\|^2\|(\Delta\bar\mu)_j\|^2}
+2\frac{(\Delta\dot\mu)_j(\bar g_j(x^*))^T(\Delta\bar\mu)_j}{\|\bar{g}_j(x^*)\|\|(\Delta\bar\mu)_j\|^2}
+\frac{(\Delta\dot\mu)_j^2}{\|(\Delta\bar\mu)_j\|^2}
=0.
$$
One can reformulate the above equality
\[
\label{keyequa}
\omega_j^2+2\theta_j\omega_j
+\frac{2c}{\sigma_j+c}
+\frac{\sigma_j-c}{\sigma_j+c}
\theta_j^2
=0
\]
with
$$
\theta_j:=\frac{(\bar g_j(x^*))^T(\Delta\bar\mu)_j}{
\|\bar g_j(x^*)\|\|(\Delta\bar\mu)_j\|}
\quad\mbox{and}
\quad
\omega_j:=\frac{(\Delta\dot\mu)_j}{\|(\Delta\bar\mu)_j\|}.
$$
Then, in order to ensure \eqref{keyequa}, its
discriminant (by viewing it as a quadratic polynomial of $\omega_j$) satisfies
$
4\theta^2_j-4\left(\frac{2c}{\sigma_j+c}
+\frac{\sigma_j-c}{\sigma_j+c}
\theta^2_j\right)\ge 0$,
which implies that
$$
\frac{c}{\sigma_j+c}\theta^2_j-\frac{c}{\sigma_j+c}
\ge 0.
$$
As a result, it holds that $\theta^2_j=1$.
Note that, if $\theta_j=1$, one has from \eqref{keyequa} that $\omega_j=-1$, or else $\omega_j=1$.
Hence, there always exists a nonzero constant $\kappa_j$ such that
\[
\label{detmuj}
(\Delta\mu)_j=\kappa_j\big(\bar{g}_j(x^*),-\dot{g}_j(x^*)\big)
\quad
\mbox{and}
\quad
\left\{
\begin{array}{ll}
\kappa_j>0,&\mbox{ if } \theta_j=1,
\\
\kappa_j<0,&\mbox{ if } \theta_j=-1.
\end{array}
\right.
\]
Note that $\dot{g}_j(x^*)=\|\bar{g}_j(x^*)\|$. Hence, from \eqref{wjsocp} and \eqref{detmuj} we know that
$$
\begin{array}{ll}
&W_j(\Delta\mu)_j
\\[2mm]
&=
\ds
\frac{1}{2}\begin{pmatrix}
\ds
\frac{2\sigma_j}{\sigma_j+c}\kappa_j\bar{g}_j(x^*)
-\frac{\sigma_j-c}{\sigma_j+c}
\frac{((\bar{g}_j(x^*))^T\bar{g}_j(x^*)}
{\|\bar{g}_j(x^*)\|^2}\kappa_j\bar{g}_j(x^*)
-\frac{-\kappa_j\dot{g}_j(x^*)}{\|\bar{g}_j(x^*)\|}\bar{g}_j(x^*)
\\[4mm]
\ds-\kappa_j\frac{(\bar{g}_j(x^*))^T\bar{g}_j(x^*)}{\|\bar{g}_j(x^*)\|}-\kappa_j \dot{g}_j(x^*)
\end{pmatrix}
\\[10mm]
&=\ds
\left(
\kappa_j \bar{g}_j(x^*),
-\kappa_j \dot{g}_j(x^*)
\right).
\end{array}
$$
Then, according to the above equation and \eqref{gsrelation}, one has that
$$\langle \xi_j, W_j(\Delta\mu)_j\rangle
=\kappa_j\big(\langle \bar{\xi}_j,\bar{g}_j(x^*)\rangle-\dot{\xi}_j \dot{g}_j(x^*)\big)= 0.
$$

{ {From all above discussion, we have proved that $\langle W_j(\Delta\mu)_j, \xi_j\rangle=0,\,\forall j=1,\ldots, s$.}} Consequently, from \eqref{deltaplus} we know that
$\|\Delta\lambda\|^2+\|\Delta\mu\|^2=0$.
This contradicts \eqref{contrac}.
That is, from $\left\langle \Delta y, V(\Delta y)\right\rangle=0$ one can only get $\Delta y=0$. Therefore, any $V\in{\mathds V}_c(\lambda^*,\mu^*)$ is  negative definite, which completes the proof.
\end{proof}

\section{The NLSDP case}
\label{sect6}
Let $\cS^p$ be the linear space of all $p\times p$ real symmetric matrices and $\cS^p_+$ be the cone of positive semidefinite matrices in $\cS^p$. We consider the NLSDP problem:
\begin{equation}
\min_x f(x)\quad\mbox{s.t.}\quad h(x)=0,\quad g(x)\in\cS_+^p,
\label{sdp}
\end{equation}
where
$f:\cX\to\Re$, $h: \cX\to \Re^{m}$ and $g: \cX\to\cS^p$ are three twice continuously differentiable functions.
Problem \eqref{sdp} is an instance of problem \eqref{op} with $\cZ=\cS^p$ and $K=\cS_+^{p}$.
We first give a quick review of some preliminary results for NLSDP problems.
One may refer to \cite{sunmor} for more details.

\subsection{Preliminaries on NLSDP}
For any two matrices $Z, Z'\in\cS^p$, the inner product of them defined by $\langle Z, Z'\rangle:= {\rm Tr}(Z^TZ')$ and its induced norm is the Frobenius norm given by $\|Z\|=\sqrt{\langle Z,Z\rangle}$.
Moreover, we write $Z_+:=\Pi_{\cS_+^p}(Z)$. Hence, it holds that
\[
\label{ortho}
Z_+\in \cS_+^p, \quad Z_+-Z\in\cS_+^p,\quad \langle Z_+, (Z_+-Z)\rangle =0.
\]
One can write its spectral decomposition as $Z=PDP^T$ with $P\in\Re^{p\times p}$ being an orthogonal matrix and $D\in\Re^{p\times p}$ being diagonal matrix whose diagonal entries are ordered by $\varrho_{1}\ge \varrho_{2},\ldots,\ge \varrho_{p}$.
We use the three index sets $\alpha,\beta$ and $\gamma$ to indicate the positive, zero, and negative eigenvalues of $Z$, respectively, i.e.,
\[
\label{abc}
\alpha:=\{i\mid \varrho_i>0\},\quad
\beta:=\{i\mid \varrho_i=0\}\quad
\mbox{and}
\quad
\gamma:=\{i\mid \varrho_i<0\},\quad
\forall i=1,\ldots,p.
\]
Then, one can write
\[
\label{sdpd}
D=\begin{pmatrix}
D_\alpha & 0&0\\
0 &D_\beta & 0\\
0 & 0 & D_\gamma
\end{pmatrix}
\quad
\mbox{and}
\quad
P=
\begin{pmatrix}
P_\alpha, P_\beta, P_\gamma
\end{pmatrix}
\]
with $P_\alpha\in\Re^{p\times |\alpha|}$,
$P_\beta\in\Re^{p\times |\beta|}$, and
$P_\gamma\in\Re^{p\times |\gamma|}$.
Moreover, we define
the matrix $U\in\cS^{p}$ by
\[
\label{defU}
U_{ij}:=\frac{\max\{\varrho_i,0\}+\max\{\varrho_j, 0\}}{|\varrho_i|+|\varrho_j|}, \quad i,j=1, \ldots, p,
\]
with the convention that $0/0=1$.
As has been shown in Sun and Sun \cite{sunsun}, the projection operator $\Pi_{\cS_+^p}$ is strongly semismooth.
Moreover, the following result in Pang et al. \cite[Lemma 11]{pangss} characterizes the Bouligand-subdifferential of this projection operator.
\begin{lemma}
\label{lem18}
Suppose that $Z\in\cS^p$ has the spectral decomposition $Z=PDP^T$ with $P$ and $D$ being partitioned by \eqref{sdpd}. Then, for any $\bW\in\partial_B\Pi_{\cS_+^p}(Z)$, there exist two index sets $\alpha'$ and $\gamma'$ that partition $\beta$, together with a matrix $\Gamma\in\Re^{|\alpha'|\times|\gamma'|}$ with entries in $[0,1]$, such that
\[
\label{patvh}
\begin{array}{l}
\bW(H)=P
\begin{pmatrix}
[P^THP]_{\alpha\alpha}& [P^THP]_{\alpha\beta} & U_{\alpha\gamma}\circ[P^THP]_{\alpha\gamma}
\\[2mm]
[P^THP]_{\alpha\beta}^T & M([P^THP]_{\beta\beta}) & 0
\\[2mm]
[P^THP]_{\alpha\gamma}^T\circ U_{\alpha\gamma}^T & 0& 0
\end{pmatrix}
P^T,\qquad
\\[2mm]
\hfill
\forall\,H\in\cS^p,
\end{array}
\]
where $``\circ"$ denotes the Hadamard product and
$$
M([P^THP]_{\beta\beta}):=
\begin{pmatrix}
([P^THP]_{\beta\beta})_{\alpha'\alpha'} & \Gamma\circ ([P^THP]_{\beta\beta})_{\alpha'\gamma'}
\\[2mm]
([P^THP]_{\beta\beta})_{\gamma'\alpha'}\circ \Gamma^T& 0
\end{pmatrix}.
$$
\end{lemma}
It is easy to see that
\[
\label{tangz}
\cT_{\cS_+^p}(Z_+)
=\{H\in \cS^p \mid P_{\bar\alpha}^T H P_{\bar \alpha}\succeq 0\},
\]
where $\bar\alpha:=\beta\cup\gamma$ and $P_{\bar\alpha}:= [P_\beta, P_\gamma]$.
From \eqref{tangz} we know that the lineality space of this set can be represented by
\[
\label{lintangz}
{\lin}\big(\cT_{\cS_+^p}(Z^+)\big) =\{H\in \cS^p \mid P_{\bar\alpha}^T H P_{\bar \alpha}= 0\}.
\]
The following definition of constraint nondegeneracy for the NLSDP problem comes from Sun \cite[Definition 3.2]{sunmor}, which is an analogue of the LICQ for NLP problems.
The relation between this condition and the concept of nondegeneracy in Robinson \cite{robmp} also has been well discussed in \cite[Section 3]{sunmor}.
\begin{definition}[Constraint nondegeneracy]
We say that a feasible point $x$ to the NLSDP problem is constraint nondegenerate if
\[
\label{constraintq}
\begin{pmatrix}
\cJ h(x)\\
\cJ g(x)
\end{pmatrix}\cX
+\begin{pmatrix}
\{0\}\\
{\lin}\big(
\cT_{\cS_+^p}(g(x))
\big)
\end{pmatrix}
=\begin{pmatrix}
\Re^{m}\\
\cS^p
\end{pmatrix}.
\]
\end{definition}

The critical cone of $\cS_+^p$ at $Z\in\cS^p$, is defined by
$$
\begin{array}{ll}
\cC(Z;S_+^p):&=\cT_{\cS_+^p}(Z_+)\cap(Z_+-Z)^\bot,
\\[2mm]
&=\{H\in\cS^p: P_\beta^THP_\beta\succeq 0,\, P_\beta^T HP_\gamma =0,\,
P_\gamma^THP_\gamma =0\}.
\end{array}
$$
Therefore,
$${\aff}\big(\cC(Z;S_+^p)\big)=\{H\in\cS^p\mid P_\beta^T HP_\gamma =0,\, P_\gamma^THP_\gamma =0\}.$$
Let $(x^*,\lambda^*,\mu^*)$ be a solution to the KKT system \eqref{kkt_2nd} of the NLSDP problem \eqref{sdp}. Then, the critical cone $\cC(x^*)$ of NLSDP at $x^*$ is defined as
$$
\cC(x^*):= \{d\in\cX\mid \cJ h(x^*)d = 0,\ \cJ g(x^*)d \in \cC(g(x^*)-\mu^*; \cS_+^p)\},
$$
Since an explicit formula for the affine hull of $\cC(x^*)$ is not readily available, an outer approximation to this affine hull, with respect to $(\lambda^*,\mu^*)$ was introduced \footnote{One may refer to Sun \cite[Section 3]{sunmor} for more information.} in Sun \cite{sunmor}, i.e.,
$$
\begin{array}{ll}
{\app}(\lambda^*,\mu^*):&
=\{d\in\cX: \cJ h(x^*)d=0, \cJ g(x^*)d \in{\aff} \cC((\lambda^*,\mu^*); \cS_+^p )\}
\\[2mm]
&=
\{d\in\cX : \cJ h(x^*)d = 0, P_\beta^T (\cJ g(x^*) d)P_\gamma=0,
P_\gamma^T(\cJ g(x^*)d)P_\gamma = 0\}.
\end{array}
$$
Moreover, for any $(\lambda^*,\mu^*)\in\cM (x^*)$, one has that
${\aff}(\cC(x^*)) \subseteq {\app}(\lambda^*,\mu^*)$, which is different from \eqref{affrel} for the NLSOCP problems.

For any given matrix $B\in\cS^p$, we define the linear-quadratic function $\Upsilon_B:\cS^p\times\cS^p\to\Re$, which is linear in the first argument and quadratic in the second argument, by
$$
\Upsilon_B(\Gamma,A):=2\langle\Gamma,AB^\dag A\rangle,\quad (\Gamma,A)\in\cS^p\times\cS^p,
$$
where $B^\dag$ is the Moore-Penrose pseudo-inverse of $B$.
The following definition of the strong second-order sufficient condition for NLSDP problems was given in Sun \cite[Definition 3.2]{sunmor}, which is an extension of
the strong second-order sufficient condition for NLP problems in Robinson \cite{rob80} to NLSDP problems.

\begin{definition}[Strong second-order sufficient condition for NLSDP]
Let $x^*$ be a stationary point of the NLSDP problem.
We say that the strong second-order sufficient condition holds at $x^*$ if
$$
\sup_{(\lambda,\mu)\in\cM(x^*)}\left\{
\langle d,\cJ^2_{xx}\cL_0(x^*,\lambda,\mu)d\rangle
-\Upsilon_{g(x^*)}(\mu, \cJ g(x^*)d)
\right\}>0,
\
\forall d\in \widehat \cC(x^*)\backslash \{0\},
$$
where
$$
\widehat \cC(x^*):=\bigcap_{(\lambda,\mu)\in\cM(x^*)}{\app}(\lambda,\mu).
$$
\end{definition}

\subsection{Main results}
Based on the above preparations, we are ready to establish the main results of this section.
\begin{theorem}
Suppose that both the constraint nondegeneracy and the strong second-order sufficient condition for problem \eqref{sdp} hold at $x^*$ $($with $(\lambda^*,\mu^*)\in\cM(x^*)$$)$.
Then, Assumptions \ref{assg} and \ref{ass:pd} hold.
\end{theorem}
\begin{proof}
The proof for Assumption \ref{assg} comes from \cite[Proposition 17]{pertub} and \cite[Proposition 4]{sun07}.
Next, we show that Assumption \ref{ass:pd} holds.

According to Assumption \ref{assg} we know that $\cM(x^*)=\{(\lambda^*,\mu^*)\}$ is a singleton. Recall that $(x^*,\lambda^*,\mu^*)$ is a solution to the KKT system of the NLSDP problem.
Then,
$$g(x^*)\in\cS_+^p\quad\mbox{and}\quad  \mu^*\in\cS_+^p.$$
If $\Xi:=g(x^*)-\mu^*$ has the spectral decomposition that $\Xi=P D P^T$ with $P$ and $D$ having the representations in \eqref{sdpd} and $\alpha,\beta,\gamma$ being defined in \eqref{abc}, one has from \eqref{ortho} that
\[
\label{gmudec}
g(x^*)=P\begin{pmatrix}
D_\alpha & 0&0\\
0 &0 & 0\\
0 & 0 & 0
\end{pmatrix}
P^T
\quad
\mbox{and}
\quad
-\mu^*=P\begin{pmatrix}
0 & 0&0\\
0 &0 & 0\\
0 & 0 & D_\gamma
\end{pmatrix}P^T.
\]
Therefore, by using \eqref{tangz} and \eqref{lintangz} one has
$$
\cT_{\cS_+^p}(g(x^*))=\{H\in\cS^p \mid [P_\beta\ P_\gamma]^T H [P_\beta\ P_\gamma]\succeq 0\},
$$
so that
\[
\label{alg:linexp}
{\lin}\big(\cT_{S_+^p}(g(x^*))\big)= \{H\in\cS^p\mid [P_\beta\ P_\gamma]^T H [P_\beta\ P_\gamma]=0\}.
\]
Define $Z:=\mu^*-cg(x^*)$ and $Q=[P_{\gamma},P_{\beta}, P_{\alpha}]$.
Then, by using \eqref{gmudec} one can get
$$
Z=
P
\begin{pmatrix}
-cD_\alpha & 0 & 0\\
0 & 0 & 0\\
0 & 0 & -D_\gamma
\end{pmatrix}
P^T
=
Q
\begin{pmatrix}
-D_\gamma & 0 & 0\\
0 & 0 & 0\\
0 & 0 & -cD_\alpha
\end{pmatrix}
Q^T.
$$
Let $\Delta\lambda\in\Re^{m}$ and $\Delta\mu\in\cS^p$ such that $(\Delta\lambda,\Delta\mu)\neq 0$.
From \eqref{patvh} and \eqref{gmudec} we know that for any
$\bW\in\partial_B\Pi_{\cS_+^p}(Z)$,

\begin{equation}
\begin{adjustbox}{width=.9999\textwidth}
$
\label{muimw}
\begin{array}{l}
\ds
\langle \Delta\mu,
(\cI_{\cS^p}- \bW)\Delta\mu
\rangle
\\
=\left\langle \Delta\mu,
Q\, Q^T(\Delta\mu)Q\, Q^T
-Q
\begin{pmatrix}
P_{\gamma}^T{\Delta\mu}P_{\gamma}& P_\gamma^T{\Delta\mu}P_{\beta} & U_{\gamma\alpha}\circ P_\gamma^T{\Delta\mu}P_{\alpha}
\\[2mm]
P_{\beta}^T{\Delta\mu}P_{\gamma}& \overline M(P_{\beta}^T{\Delta\mu}P_{\beta}) & 0\\[2mm]
P_{\alpha}^T{\Delta\mu}P_{\gamma}\circ U_{\gamma\alpha}^T & 0& 0
\end{pmatrix} Q^T
\right\rangle
\\[12mm]
=\left\langle Q^T\Delta\mu Q,
Q^T(\Delta\mu)Q\,
-
\begin{pmatrix}
P_{\gamma}^T{\Delta\mu}P_{\gamma}& P_\gamma^T{\Delta\mu}P_{\beta} & U_{\gamma\alpha}\circ P_\gamma^T{\Delta\mu}P_{\alpha}
\\[2mm]
P_{\beta}^T{\Delta\mu}P_{\gamma}& \overline M(P_{\beta}^T{\Delta\mu}P_{\beta}) & 0\\[2mm]
P_{\alpha}^T{\Delta\mu}P_{\gamma}\circ U_{\gamma\alpha}^T & 0& 0
\end{pmatrix}
\right\rangle
\\[12mm]
=\left\langle Q^T\Delta\mu Q,
\begin{pmatrix}
0& 0 & P_\gamma^T{\Delta\mu}P_{\alpha}-U_{\gamma\alpha}\circ P_\gamma^T{\Delta\mu}P_{\alpha}
\\[2mm]
0& (I_{\beta}-\overline M)(P_{\beta}^T{\Delta\mu}P_{\beta}) & P_{\beta}^T{\Delta\mu}P_{\alpha}\\[2mm]
P_{\alpha}^T{\Delta\mu}P_{\gamma}-
P_{\alpha}^T{\Delta\mu}P_{\gamma}\circ U_{\gamma\alpha}^T & P_{\alpha}^T{\Delta\mu}P_{\beta}&
P_{\alpha}^T{\Delta\mu}P_{\alpha}
\end{pmatrix}
\right\rangle,
\end{array}
$
\end{adjustbox}
\end{equation}
where $I_{\beta}\in\Re^{|\beta|\times|\beta|}$ is the identity matrix and $\overline M$ has the same structure and properties as the matrix $M$ in Lemma \ref{lem18}.
Then, by \eqref{defU} and the fact that $\Delta\mu$ is arbitrarily given, we can conclude that $(\cI_{\cS^p}- \bW)$ is positive semidefinite, where $\cI_{\cS^p}$ is the identity operator on $\cS^p$.
Since Assumption \ref{assg} holds, any $\bV\in{\mathds V}_c(\lambda^*,\mu^*)$ can be represented by the right-hand side of \eqref{v} with $\cA_c(\lambda^*,\mu^*,\bW)$ being nonsingular.
Therefore,
$$
\begin{array}{ll}
-\langle (\Delta\lambda,\Delta\mu), \bV(\Delta\lambda,\Delta\mu)\rangle
\\[2mm]
= \Big\langle
\nabla h(x^*)(\Delta\lambda)+\nabla g(x^*) \bW(\Delta\mu)
,
\\[2mm]
\quad(\cA_c(\lambda^*,\mu^*,\bW))^{-1}
\big(\nabla h(x^*)(\Delta\lambda)+\nabla g(x^*) \bW(\Delta\mu)\big)\Big\rangle
+
\frac{1}{c}
\left
\langle \Delta\mu,
(\cI_{\cS^p}- \bW)\Delta\mu
\right
\rangle.
\end{array}
$$
Therefore, if $\langle (\Delta\lambda,\Delta\mu), \bV(\Delta\lambda,\Delta\mu)\rangle=0$, one has that
\[
\label{eq002}
\langle \Delta\mu,
(\cI_{\cS^p}- \bW)\Delta\mu
\rangle
=0,
\]
and
\[
\label{eq001}
\nabla h(x^*)(\Delta\lambda)+\nabla g(x^*) \bW(\Delta\mu)=0.
\]
Since the constraint nondegeneracy \eqref{constraintq} holds for the NLSDP problem at $x^*$, there exists a vector $d\in \cX$ and a matrix $S\in{\lin}\big(
\cT_{\cS_+^p}(g(x^*))
\big)
$ such that
$$
\cJ h(x^*)d=\Delta\lambda
\quad
\mbox{and}\quad
\cJ g(x^*)d+S=\Delta\mu.
$$
Therefore, one can get from \eqref{eq002} and \eqref{eq001} that
$$
\begin{array}{l}
\|\Delta\lambda\|^2+\|\Delta\mu\|^2=
\langle \Delta\lambda,\Delta\lambda\rangle+\langle\Delta\mu, \Delta\mu\rangle
\\[2mm]
=\langle \Delta\lambda,\Delta\lambda\rangle
+\langle \bW\Delta\mu, \Delta\mu\rangle
+\langle (\cI_{\cS^p}-\bW)\Delta\mu, \Delta\mu\rangle
\\[2mm]
=\langle \Delta\lambda, \cJ h(x^*)d\rangle
+\langle \bW\Delta\mu, \cJ g(x^*)d+S\rangle
+\langle (\cI_{\cS^p}-\bW)\Delta\mu, \Delta\mu\rangle
\\[2mm]
=\langle \nabla h(x^*)(\Delta\lambda)
+\nabla g(x^*)\bW(\Delta\mu), d\rangle
+\langle (\cI_{\cS^p}-\bW)\Delta\mu, \Delta\mu\rangle
+\langle \bW\Delta\mu, S\rangle
\\[2mm]
=\langle \bW\Delta\mu, S\rangle.
\end{array}
$$
Since $S\in{\lin}\big(
\cT_{\cS_+^p}(g(x^*))
\big)
$, we know from \eqref{alg:linexp} that
$[P_\beta\ P_\gamma]^T S [P_\beta\ P_\gamma]=0$.
Therefore, it holds that

$$
\begin{array}{lll}
\langle\bW\Delta\mu, S\rangle
&=&
\left\langle
\begin{pmatrix}
P_{\gamma}^T{\Delta\mu}P_{\gamma}& P_\gamma^T{\Delta\mu}P_{\beta} & U_{\gamma\alpha}\circ P_\gamma^T{\Delta\mu}P_{\alpha}
\\[1mm]
P_{\beta}^T{\Delta\mu}P_{\gamma}& \overline M(P_{\beta}^T{\Delta\mu}P_{\beta}) & 0\\[1mm]
P_{\alpha}^T{\Delta\mu}P_{\gamma}\circ U_{\gamma\alpha}^T & 0& 0
\end{pmatrix}
\right.,
\\[10mm]
&&\qquad\qquad\qquad
\left.
\begin{pmatrix}
0 &0 & P^T_\gamma S P_\alpha
\\[1mm]
0 &0 & P^T_\beta S P_\alpha
\\[1mm]
 P^T_\alpha S P_\gamma& P^T_\alpha S P_\beta& P^T_\alpha S P_\alpha
\end{pmatrix}
\right\rangle
\\[10mm]
&= &\langle U_{\gamma\alpha}\circ P_\gamma^T{\Delta\mu}P_{\alpha}, P^T_\gamma S P_\alpha \rangle
+\langle P_{\alpha}^T{\Delta\mu}P_{\gamma}\circ U_{\gamma\alpha}^T, P^T_\alpha S P_\gamma \rangle.
\end{array}
$$
Suppose that $\langle (\Delta\lambda,\Delta\mu), \bV(\Delta\lambda,\Delta\mu)\rangle=0$.
From \eqref{muimw} and \eqref{eq002} we know that
$$
\langle
P_\gamma^T{\Delta\mu}P_{\alpha},
P_\gamma^T{\Delta\mu}P_{\alpha}-U_{\gamma\alpha}\circ P_\gamma^T{\Delta\mu}P_{\alpha}
\rangle=0.
$$
Then, by using \eqref{defU} and \eqref{gmudec} one can get
$\|P_\gamma^T{\Delta\mu}P_{\alpha}\|=0$. Hence $\langle\bW\Delta\mu, S\rangle=0$.
Consequently, it holds that
$$\|\Delta\lambda\|^2+\|\Delta\mu\|^2=0.$$
This implies that if $\langle (\Delta\lambda,\Delta\mu), \bV(\Delta\lambda,\Delta\mu)\rangle=0$, one must have
$\Delta\lambda=0$ and $\Delta\mu=0$. Therefore, any $\bV\in{\mathds V}_c(\lambda^*,\mu^*)$ is negative definite. This completes the proof.
\end{proof}

\section{Numerical illustration}
\label{secnum}
In this section, we use a simple numerical example to test whether the numerical performance of the second-order augmented Lagrangian method can achieve a superlinear convergence.
For comparison, the classic augmented Lagrangian method with the linear convergence rate is also implemented.

Consider the following problem
$$
\min_{(x_1,x_2)\in\Re^2}
\left\{ \frac{1}{2}(x_1-1)^2+\frac{1}{2}(x_2-2)^2
\ \mid\
x_1=0,\ x_2\ge 0
\right\}.
$$
It is easy to verify the solution of this problem is given by  $(x_1^*,x_2^*)=(0,2)$, and both the LICQ (Definition \ref{deflicq}) and the strong second-order sufficient condition (Definition \ref{defssoscnlp}) are satisfied at this point with the corresponding multipliers given by $(\lambda^*,\mu^*)=(-1,0)$.
We set the penalty parameter $c=1$, and implement both the augmented Lagrangian method and the second-order augmented Lagrangian method for this problem with the initial multipliers \footnote{The initial multiplier is not intentionally chosen. One can observe the same numerical behavior from other initial points.} $(\lambda^0,\mu^0):=(100,100)$.

The numerical experiment is conducted by running {\sc Matlab} R2020a on a MacBook Pro (macOS 11.2.3 system) with one Intel 2.7 GHz Quad-Core i7 Processor and 16 GB RAM.
We measure the error at the $k$-th iteration via
$$\eta^k:=\|(\lambda^k,\mu^k)-(\lambda^*,\mu^*)\|.
$$
The algorithms are terminated if $\eta^k\le 10^{-20}$.
Moreover, we set $\eta^k$ as $10^{-50}$ if $\eta^k\le 10^{-50}$, so that a valid value can be obtained by taking the logarithm of the error when it is too close to zero.

Figure \ref{fig11} clearly shows the behaviors of the two methods. The left part shows how the errors of the two methods varies with respect to the iteration number, while the right part of this figure, by taking logarithm of the errors, explicitly shows the linear convergence of the classic augmented Lagrangian method, as well as the superlinear convergence of the second-order method of multipliers.

\begin{center}
\begin{figure}[h]
\caption{
\label{fig11}
The error (left) and the logarithm of the error (right) with respect to the iteration}
\includegraphics[width=0.48\textwidth]{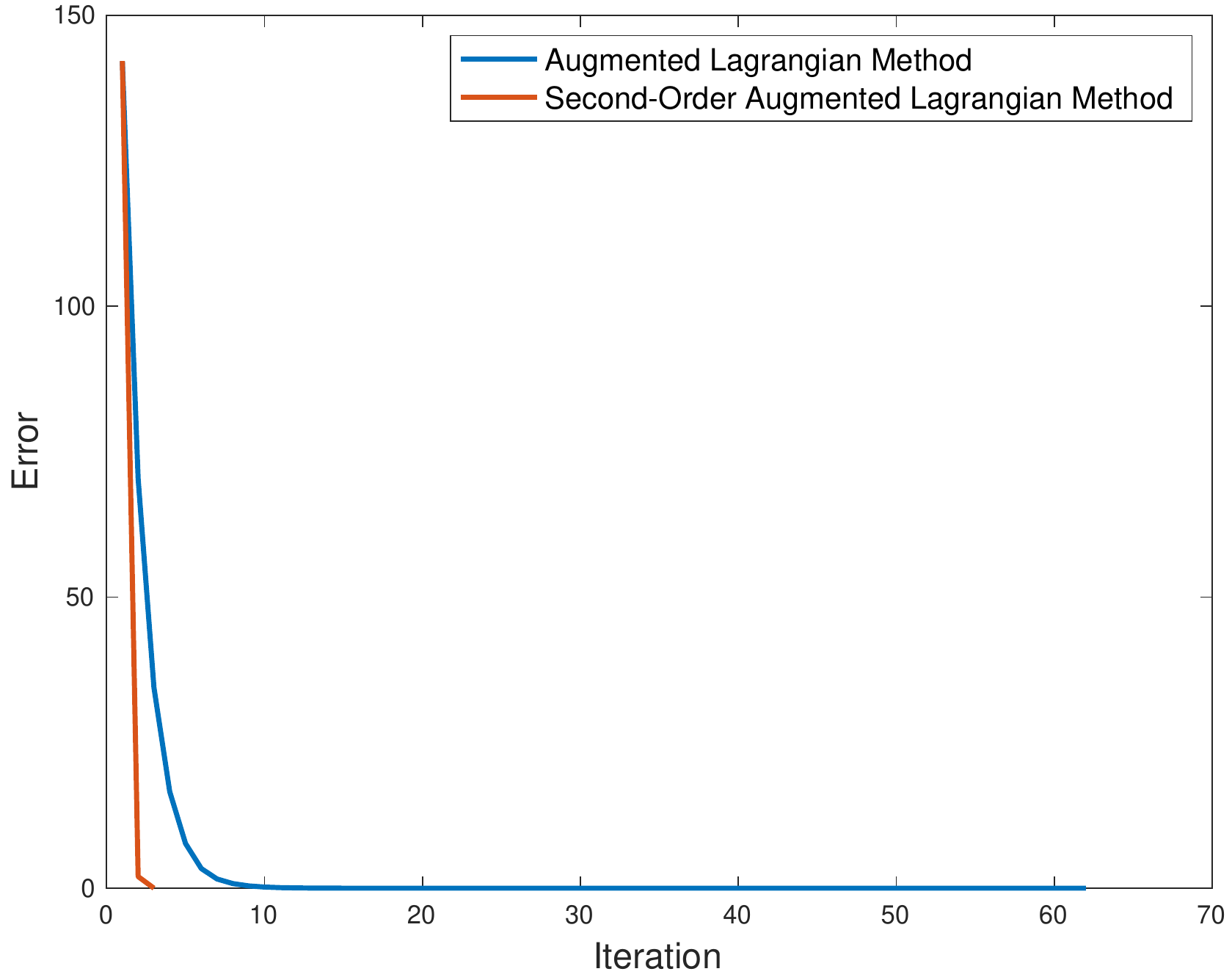}
\includegraphics[width=0.48\textwidth]{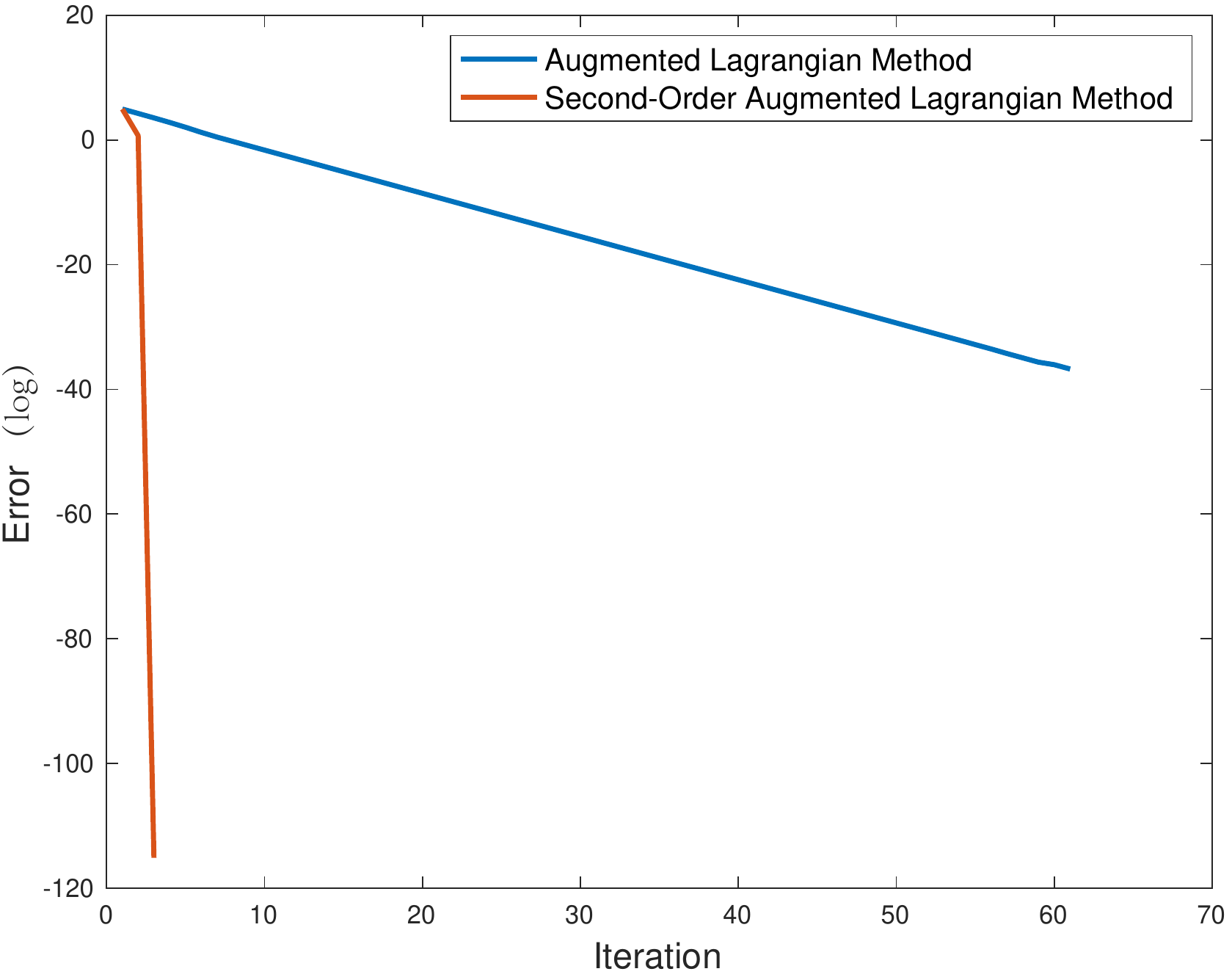}
\end{figure}
\end{center}

\color{black}

\section{Conclusions}
\label{sec7}
In this paper, we have proposed a second-order method of multipliers for solving the NLP, the NLSOCP and the NLSDP problems. The proposed method is a combination of the classic method of multipliers and a specially designed nonsmooth Newton method for updating the multipliers.
The local convergence, as well as the superlinear convergence of this second-order method was established under the constraint nondegeneracy (or LICQ) together with a strong second-order sufficient condition.
We re-emphasize that, the conditions we used in this paper for proving the superlinear rate convergence of the second-order method of multipliers are exactly those used in the literature for establishing the linear rate convergence (or asymptotically superlinear if the penalty parameter goes to infinity) of the classic method of multipliers without assuming the strictly complementarity.
Besides, a simple numerical test was conducted to demonstrate the superlinear convergence behavior of the method studied in this paper.

Since the proposed method is based on both the augmented Lagrangian method and generalized Newton method, we are not able to establish any global convergence property of this method in general. However, to make the method being implementable, globalization techniques should be further studied, together with the computational studies on solving the subproblems and updating the multipliers.


\bigskip
\begin{appendix}
\section{Proof of Proposition \ref{ppvy}}
\label{apppp2}
The following result is a consequence of Qi and Sun \cite[Theorem 2.3]{qi93}.
\begin{lemma}
\label{lemma:semi}
The locally Lipschitz continuous function $F: O\to \cU$ is semismooth at $u\in O$ if and only if
for any sequence $\{u^k\}\subset O$ converging to $u$ with $V^k\in\partial F(u^k)$, it holds that
$F\,'(u;u^k-u)-V^k(u^k-u)=o(\|u^k-u\|)$.
\end{lemma}

Now we start to prove Proposition \ref{ppvy}.

\begin{proof}[Proof of Proposition \ref{ppvy}]
Since Assumption \ref{assg} holds and $c\ge\bar c$, one can get the two constants $\varepsilon$ and $\delta_0$, as well as the locally Lipschitz function $\bfx_c(\cdot)$, from Proposition \ref{summary}.
For convenience, we denote $y:=(\lambda,\mu)\in\Re^m\times\cZ$.
Suppose that $y\in {\mathds B}_{\delta_0}(\lambda^*,\mu^*)$.
From \eqref{alf}, \eqref{laf}, \eqref{xcy} and \eqref{multiplus} one can get
$$
\nabla_x\cL_0(\bfx_c(y),\lambda_c(y),\mu_c(y))
=\nabla_x\cL_c(\bfx_c(y),\lambda,\mu)=0.
$$
Then, by taking the directional derivative of $\nabla_x\cL_0(\bfx_c(y),\lambda_c(y),\mu_c(y))$ along with the given direction $\Delta y:=(\Delta\lambda,\Delta\mu)\in\Re^{m}\times\cZ$ and using \eqref{multiplus} one can see that
\[
\label{n2}
\begin{array}{ll}
0=&
\nabla_{xx}^{2}\cL_0(\bfx_c(y),\lambda_c(y),\mu_c(y))\bfx_c'(y;\Delta y)
\\[2mm]
&
-\nabla h(\bfx_c(y))(\Delta\lambda)
+c\nabla h(\bfx_c(y))\cJ h(\bfx_c(y))\bfx_c'(y;\Delta y)
\\[2mm]
&-\nabla g(\bfx_c(y))\Pi_K'(\mu-cg(\bfx_c(y));\Delta\mu-c\cJ g(\bfx_c(y))\bfx_c'(y;\Delta y)).
\end{array}
\]
On the other hand, since the projection operator $\Pi_K(\cdot)$ is strongly semismooth, we know that there always exists a linear operator $\widehat W\in\partial_B\Pi_K(\mu-cg(\bfx_c(y)))$ such that
\[
\label{p2}
\Pi_K'\left(\mu-cg(\bfx_c(y));\Delta\mu-c\cJ g(\bfx_c(y))\bfx_c'(y;\Delta y)\right)
=\widehat W\left(\Delta\mu-c\cJ g(\bfx_c(y))\bfx_c'(y;\Delta y)\right).
\]
Therefore, by combining (\ref{n2}), (\ref{p2}) and (\ref{a2}) together, one can get
\[
\label{acposi}
\cA_c(\lambda,\mu,\widehat W)\bfx_c'(y;\Delta y)
=\nabla h(\bfx_c(y))(\Delta\lambda)+\nabla g(\bfx_c(y))\widehat W(\Delta\mu).
\]
From \cite[Proposition 1]{meng} we know that the range of $\partial_{B}\Pi_{K}(\cdot)$ is bounded.
Then, by \cite[proposition 5.51(b)]{va} and \cite[Proposition 5.52(b)]{va} we know that the composite mapping $\cA_c(\lambda,\mu,\partial_{B}\Pi_{K}(\mu-cg(y)))$ is outer semicontinuous in ${\mathds B}_{\delta_0}(\lambda^*,\mu^*)$.
Note that the set $\cA_c(\lambda^*,\mu^*,\partial_{B}\Pi_{K}(\mu^*-cg(\lambda^*,\mu^*)))$ is bounded and bounded away from zero.
Hence, by \eqref{alimit} and Assumption \ref{assg}, we know from \cite[Proposition 5.12(a)]{va} that there exists a positive constant $\delta_1\le\delta_0$ such that $\cA_c(\lambda,\mu,W)$ is always positive definite if $W\in\partial_B\Pi_K(\mu-cg(\bfx_c(y)))$ with $y=(\lambda,\mu)\in{\mathds B}_{\delta_1}(\lambda^*,\mu^*)$.
Consequently, in this case, we know from \eqref{acposi} that
\[
\label{xcpy}
\bfx_c'(y;\Delta y)=(\cA_c(\lambda,\mu,\widehat W))^{-1}\left(\nabla h(\bfx_c(y))(\Delta\lambda)+\nabla g(\bfx_c(y))\widehat W(\Delta\mu)\right).
\]
We use $D_{\nabla\vartheta_c}$ to denote the set of all the points in ${\mathds B}_{\delta_1}(\lambda^*,\mu^*)$ where $\nabla\vartheta_c(\cdot)$ is Fr\'echet-differentiable.
Since $\nabla\vartheta_c(\cdot)$ is semismooth in ${\mathds B}_{\delta_1}(\lambda^*,\mu^*)$, it holds that for any $y=(\lambda,\mu)\in D_{\nabla\vartheta_c}$,
\begin{equation}
\label{nabla2}
\nabla^2\vartheta_c(y)(\Delta y)
=
\left(
\begin{matrix}
-\cJ h(\bfx_c(y))\bfx_c'(y;\Delta y)\\[1mm]
-c^{-1}\Delta\mu+c^{-1}\Pi_K'\big(\mu-cg(\bfx_c(y));\Delta\mu-c\cJ g(\bfx_c(y))\bfx_c'(y;\Delta y)\big)
\end{matrix}
\right).
\end{equation}
Thus, by combining \eqref{p2}, \eqref{xcpy} and \eqref{nabla2} together, one can see from \eqref{v} that
$$
\nabla^2\vartheta_c(y)(\Delta y)
\\[2mm]
\in
\left\{V(\Delta y)\mid V\in{\mathds V}_c(y)
\right\},
\quad\forall y=(\lambda,\mu)\in D_{\nabla\vartheta}.
$$
Since $\bfx_c(\cdot)$ is locally Lipschitz continuous and $\partial_{B}\Pi_{K}(\cdot)$ is bounded-valued and outer semicontinuous, by \cite[Propositions 5.51(b) \& 5.52(b)]{va} we know that ${\mathds V}_c$ is compact-valued and outer semicontinuous.
Hence, for any $(\lambda,\mu)\in{\mathds B}_{\delta_1}(\lambda^*,\mu^*)$, one has that
$\partial_B(\nabla\vartheta_c)(\lambda,\mu)
\subseteq
{\mathds V}_{c}(\lambda,\mu)$. This completes the proof.
\end{proof}

\section{Proof of Theorem \ref{th31}}
\label{ap3}
\begin{proof}
Since Assumption \ref{assg} holds and $c\ge\bar c$, one can get the two constants $\varepsilon$ and $\delta_0$, as well as the locally Lipschitz function $\bfx_c(\cdot)$, via Proposition \ref{summary}.
Meanwhile, from Proposition \ref{ppvy} we know that there exists a constant $\delta_1\le\delta_0$ such that the mapping ${\mathds V}_c$ defined by \eqref{v} is well-defined, compact-valued and outer semicontinuous in ${\mathds B}_{\delta_1}(\lambda^*,\mu^*)$.
Since the sequence $\{(\lambda^k,\mu^k)\}$ converges to $(\lambda^*,\mu^*)$, one always has that $(\lambda^k,\mu^k)\in{\mathds B}_{\delta_1}(\lambda^*,\mu^*)$ for all sufficiently large $k$.
Therefore, the set ${\mathds V}_c(\lambda^{k},\mu^k)$ is nonempty and compact if $k$ is sufficiently large.
Moreover, according to Assumption \ref{ass:pd} and \cite[Proposition 5.12(a)]{va}
we know that there exists a positive constant $\delta_2\le\delta_1$ such that every element in ${\mathds V}_c(\lambda,\mu)$ is negative definite whenever $(\lambda,\mu)\in {\mathds B}_{\delta_2}(\lambda^*,\mu^*)$.

Next, we prove \eqref{gt}.
For all sufficiently large $k$, with the convention that $x^{k+1}:=\bfx_c(\lambda^k,\mu^k)$, one can see from \eqref{v} that there exists a linear operator $W^{k}\in\partial_B\Pi_K(\mu^{k}-cg(x^{k+1}))$ such that
\[
\label{vy}
\begin{array}{ll}
V^k
=&-\begin{pmatrix}
\cJ h(x^{k+1})\\[0mm]
W^k\cJ g(x^{k+1})
\end{pmatrix}
\left(\cA_c(\lambda^k,\mu^k,W^k)\right)^{-1}
\left(
\nabla h(x^{k+1}),
\nabla g(x^{k+1})W^k
\right)
\\[4mm]
&-c^{-1}\left(
\begin{matrix}
0&0\\[0mm]
0& \cI_{\cZ}-W^k
\end{matrix}
\right).
\end{array}
\]
Meanwhile, one can see from \eqref{kkt_2nd} and Lemma \ref{lemma:nal} that
\begin{equation}
\label{naby}
\begin{array}{l}
\nabla\vartheta_c(\lambda^k,\mu^k)-\nabla\vartheta_c(\lambda^*,\mu^*)
\\[2mm]
=\left(\begin{matrix}
-h(x^{k+1})+h(x^*)
\\[2mm]
-c^{-1}(\mu^k-\mu^*)+c^{-1}\Pi_K\big(\mu^{k}-cg(x^{k+1}))-c^{-1}\Pi_K(\mu^*-cg(x^*)\big)
\end{matrix}\right).
\end{array}
\end{equation}
For convenience, we denote $y^k:=(\lambda^k,\mu^k)$ and $y^*:=(\lambda^*,\mu^*)$.
Then, by \eqref{vy} and \eqref{naby} we know that \eqref{gt} holds if and only if for all $W^{k}\in\partial_B\Pi_K(\mu^{k}-cg(x^{k+1}))$ one has that, with $k\to\infty$,
\[
\label{r1}
\begin{array}{l}
h(x^*)-h(x^{k+1})
\\[1mm]
+\cJ h(x^{k+1})
(\cA_c(\lambda^k,\mu^k,W^k))^{-1}
\left(
\nabla h(x^{k+1}),
\nabla g(x^{k+1})W^k
\right)(y^k-y^*)
\qquad
\\[1mm]
\hspace*{\fill}
=o(\|y^k-y^*\|)
\end{array}
\]
and
\begin{multline}
\label{r2}
W^k\cJ g(x^{k+1})
(\cA_c(\lambda^k,\mu^k,W^k))^{-1}
\left(
\nabla h(x^{k+1}),
\nabla g(x^{k+1})W^k
\right)(y^k-y^*)
\\[1mm]
+c^{-1}\Pi_K(\mu^k-cg(x^{k+1}))-c^{-1}\Pi_K(\mu^*-cg(x^*))-c^{-1}W^k
(\mu^k-\mu^*)
\\[1mm]
\hspace*{\fill}
 =o(\|y^k-y^*\|).
\end{multline}
We first veirfy \eqref{r1}.
We know from \eqref{p2} and \eqref{xcpy} that, when $k$ is sufficiently large,
\[
\label{xpie}
\begin{array}{ll}
\bfx_c'(y^k;y^k-y^*)
\\[2mm]
=\big(\cA_c(\lambda^k,\mu^k,\widehat W^{k})\big)^{-1}\big(\nabla h(x^{k+1})(\lambda^{k}-\lambda^*)+\nabla g(x^{k+1})\widehat W^{k}(\mu^{k}-\mu^*)\big),
\end{array}
\]
where the linear operator $\widehat W^{k}\in\partial_B\Pi_K(\mu^{k}-cg(x^{k+1}))$ is chosen such that
\[
\label{eq:hatw}
\begin{array}{ll}
\displaystyle
\Pi_K'\big(\mu^{k}-cg(x^{k+1});\mu^{k}-\mu^*-c\cJ g(x^{k+1}) \bfx_c'(y^k;y^k-y^*)
\big)
\\[2mm]
\displaystyle
=\widehat W^{k}\big(\mu^{k}-\mu^*-c\cJ g(x^{k+1})
\bfx_c'(y^k;y^k-y^*)\big).
\end{array}
\]
As has been discussed in Appendix  \ref{apppp2}, the range of $\partial_{B}\Pi_{K}(\cdot)$ is bounded, so that the sequence $\{\widehat W^{k}\}$ is also bounded.
Moreover, the set $\cA_c(\lambda^*,\mu^*,\partial_{B}\Pi_{K}(\mu^*-cg(\lambda^*,\mu^*)))$ is bounded and bounded away from zero,
while the composite mapping $\cA_c(\lambda,\mu,\partial_{B}\Pi_{K}(\mu-cg(y)))$ is outer semicontinuous in ${\mathds B}_{\delta_0}(\lambda^*,\mu^*)$.
Note that the functions $g$ and $h$ are twice continuously differentiable, the sequences $\{\nabla h(x^{k+1})\}$ and $\{\nabla g(x^{k+1})\}$ are bounded.
Then, we know from \eqref{xpie} that there exists a constant $\rho_{1}>0$ such that, for all sufficiently large $k$,
$\|\bfx_c'(y^k;y^k-y^*)\|\le\rho_{1}\|y^k-y^*\|$.

Since $\Pi_K(\cdot)$ is strongly semismooth and $W^{k}\in\partial_B\Pi_K(\mu^{k}-cg(x^{k+1}))$,
one can see from \eqref{eq:hatw} that, when $k\to\infty$ and $k$ is sufficiently large,
$$
\begin{array}{l}
\Pi_K'\left(\mu^{k}-cg(x^{k+1});\mu^{k}-\mu^*-c\cJ g(x^{k+1})\bfx_c'(y^k;y^k-y^*)
\right)
\\[2mm]
\qquad
-W^k(\mu^{k}-\mu^*-c\cJ g(x^{k+1})
\bfx_c'(y^k;y^k-y^*))
\\[2mm]
=(\widehat W^k-W^k)\left(\mu^{k}-\mu^*-c\cJ g(x^{k+1})
\bfx_c'(y^k;y^k-y^*)\right)
\\[2mm]
=(\widehat W^k-W^k)\left(\mu^k-cg(x^{k+1})-\mu^*+cg(x^*)\right)
\\[2mm]
\quad+c(\widehat W^k-W^k)\left( g(x^{k+1})-g(x^*)-\cJ g(x^{k+1})\bfx_c'(y^k;y^k-y^*)
) \right)
\\[2mm]
=o\left(\|\mu^k-cg(\bfx_c(y^k))-\mu^*+cg(x^*)\|\right)
\\[1mm]
\quad+c(\widehat W^k-W^k)
\big(g(\bfx_c(y^k))-g(\bfx_c(y^*))-\cJ g(x^{k+1})\bfx_c'(y^k;y^k-y^*)\big)
\\[2mm]
=o(\|y^k-y^*\|),
\end{array}
$$
where the penultimate equality is due to Lemma \ref{lemma:semi}, and the
last inequality comes from the fact that the function
$g(\bfx_c(\cdot))$ is semismooth (hence directionally differentiable) around $y^*$.
Then, from the fact that $\{\nabla g(x^{k+1})\}$ is bounded, we can get from the above equation that
\[
\label{oterm1}
\begin{array}{l}
\nabla g(x^{k+1})\Pi_K'\left(\mu^k-cg(x^{k+1});\mu^k-\mu^*-c\cJ g(x^{k+1})\bfx_c'(y^k;y^k-y^*)\right)
\\[2mm]
\qquad
+c\nabla g(x^{k+1})W^k\cJ g(x^{k+1})
\bfx_c'(y^k;y^k-y^*)
-\nabla g(x^{k+1})W^k(\mu^k-\mu^*)
\qquad
\\[2mm]
\hspace*{\fill}
=o(\|y^k-y^*\|).
\end{array}
\]
Note that \eqref{n2} in  Appendix \ref{apppp2} holds for $y\in {\mathds B}_{\delta_2}(\lambda^*,\mu^*)$.
Then, by using \eqref{n2} and \eqref{a2} together we can get that for $y\in {\mathds B}_{\delta_2}(\lambda^*,\mu^*)$ and $W\in\partial_B\Pi_K(\mu -cg(\bfx_c(y)))$,
\[
\label{oterm2}
\begin{array}{ll}
\displaystyle \cA_c(\lambda,\mu,W)\bfx_c'(y;\Delta y)
\\[2mm]
=
\nabla h(\bfx_c(y))(\Delta\lambda)
+c\nabla g(\bfx_c(\lambda,\mu))W\cJ g(\bfx_c(\lambda,\mu))\bfx_c'(y;\Delta y)
\\[2mm]
\qquad
+\nabla g(\bfx_c(y))\Pi_K'(\mu-cg(\bfx_c(y));\Delta\mu-c\cJ g(\bfx_c(y))\bfx_c'(y;\Delta y)).
\end{array}
\]
Then, by taking \eqref{oterm1} into \eqref{oterm2} one can get that for all sufficiently large $k$ with $W^{k}\in\partial_B\Pi_K(\mu^{k}-cg(x^{k+1}))$,
$$
\begin{array}{ll}
\cA_c(\lambda^k,\mu^k,W^k)\bfx_c'(y^k;y^k-y^*)
\\[2mm]
=\nabla h(x^{k+1})(\lambda^k-\lambda^*)
+\nabla g(x^{k+1})W^k(\mu^k-\mu^*)
+o(\|y^k-y^*\|),
\end{array}
$$
which, together with Assumption \ref{assg} the fact
the mapping $\cA_c(\lambda,\mu,\partial_{B}\Pi_{K}(\mu-cg(\bfx_c(y))))$
is outer semicontinuous in ${\mathds B}_{\delta_0}(\lambda^*,\mu^*)$, implies that for all sufficiently large $k$, as $k\to\infty$,
\[
\label{result1}
\begin{array}{ll}
\bfx_c'(y^k;y^k-y^*)
\\[2mm]
=(\cA_c(\lambda^k,\mu^k,W^k))^{-1}
\left(\nabla h(x^{k+1}),
\nabla g(x^{k+1})W^k\right)
(y^k-y^*)
+o(\|y^k-y^*\|).
\end{array}
\]
Since $h(\bfx_c(\cdot))$ is locally semismooth, we have that for sufficiently large $k\to\infty$,
\begin{equation}
\label{hbarx}
h(x^*)=h(x^{k+1})-\cJ h(x^{k+1})\bfx_c'(y^k;y^k-y^*)+o(\|y^k-y^*\|).
\end{equation}
Now, by substituting \eqref{result1} in \eqref{hbarx} and noting that the sequence $\{\cJ h(x^{k+1})\}$ is bounded, we can get \eqref{r1}.
Next, we verify \eqref{r2}.
Note that for sufficiently large $k$ with $k\to\infty$, it holds that
\[
\label{ee1}
\begin{array}{l}
c^{-1}(\Pi_K(\mu^k-cg(x^{k+1}))-\Pi_K(\mu^*-cg(x^*)))\\[1mm]
=c^{-1}\Pi_{K}'\left(\mu^{k}-cg(x^{k+1}),\mu^{k}-\mu^*-cg(x^{k+1})+cg(x^*)\right)
+o(\|y^k-y^*\|),
\end{array}
\]
and
\[
\label{ee2}
-cg(x^{k+1})+cg(x^*)
=-c\cJ g(x^{k+1})\bfx_c'(y^k;y^k-y^*)+o(\|y^k-y^*\|).
\]
Meanwhile, one can choose for each $k$ the linear operator $\widetilde W^{k}\in\partial_{B}\Pi_{K}(\mu^{k}-cg(x^{k+1}))$ such that
\begin{equation}
\label{ee3}
\begin{array}{l}
\widetilde W^{k}\big(\mu^{k}-\mu^*-cg(x^{k+1})+cg(x^*)\big)
\\[2mm]
=\Pi_{K}'\big(\mu^{k}-cg(x^{k+1});\mu^{k}-\mu^*-cg(x^{k+1})+cg(x^*)\big).
\end{array}
\end{equation}
Then by combining \eqref{ee1}, \eqref{ee2} and \eqref{ee3} together, we can get that
\begin{equation}
\label{eq2main}
\begin{array}{l}
c^{-1}\left(\Pi_K(\mu^k-cg(x^{k+1}))-\Pi_K(\mu^*-cg(x^*))\right)
-c^{-1}W^k(\mu^k-\mu^*)
\\[2mm]\quad
+W^k\cJ g(x^{k+1}) \bfx_c'(y^k;y^k-y^*)\\[2mm]
=c^{-1}\widetilde W^{k}
\left(\mu^{k}-\mu^*-cg(x^{k+1})+cg(x^*)\right)
\\[2mm]
\qquad
-c^{-1}W^k
\left(\mu^{k}-\mu^*-cg(x^{k+1})+cg(x^*)\right)
+o(\|y^k-y^*\|)\\[2mm]
=c^{-1}(\widetilde W^{k}-W^{k})\left(\mu^{k}-\mu^*-cg(x^{k+1})+cg(x^*)\right)+o(\|y^k-y^*\|)
\\[2mm]
=o(\|y^k-y^*\|),
\end{array}
\end{equation}
where the last inequality comes from Lemma \ref{lemma:semi} and the fact that $\widetilde W^{k},W^{k}\in\partial_{B}\Pi_{K}(\mu^{k}-cg(x^{k+1}))$.
Note that \eqref{result1} implies that
\begin{equation}
\label{result2}
\begin{array}{l}
W^k\cJ g(x^{k+1})\bfx_c'(y^k;y^k-y^*)
\\[2mm]
=W^k\cJ g(x^{k+1})(\cA_c(\lambda^k,\mu^k,W^k))^{-1}
\left(\nabla h(x^{k+1}),
\nabla g(x^{k+1})W^k\right)
(y^k-y^*)
+o(\|y^k-y^*\|).
\end{array}
\end{equation}
Thus, by substituting \eqref{result2} in \eqref{eq2main} we can get \eqref{r2}, which completes the proof.
\end{proof}

\end{appendix}

\end{document}